\documentclass[reqno,11pt]{amsart}
\usepackage[numbers]{natbib}
\usepackage[french,english]{babel}
\usepackage[utf8]{inputenc}
\usepackage[T1]{fontenc}
\usepackage{lmodern}
\usepackage{amsmath}
\usepackage{amssymb}
\usepackage{comment}
\usepackage{bm}
\usepackage{mathrsfs}
\usepackage{amsthm}
\usepackage{enumitem}
\usepackage{dsfont}

\usepackage{array}
\usepackage{tabularx}
\usepackage{booktabs}
\usepackage{ragged2e}

\usepackage{eurosym} 
\usepackage{stmaryrd} 

\usepackage{tikz}
\usetikzlibrary{arrows.meta}
\usepackage{graphicx}
\usepackage{multirow} 
\usepackage{float}    
\usepackage{adjustbox}

\usepackage{hyperref}
\hypersetup{
  colorlinks=true,
  breaklinks=true,
  urlcolor= blue,
  linkcolor= blue,
  citecolor= magenta,
  pdfstartview = FitH,
  bookmarksopen = true
}

\usepackage{listings}
\usepackage{xcolor}
\usepackage{framed}

\usepackage{fullpage} 

\numberwithin{equation}{section}
\setlist{leftmargin=2.2em}

\newcommand{\E}{\mathbb{E}}
\newcommand{\Prob}{\mathbb{P}}
\newcommand{\Q}{\mathbb{Q}}
\newcommand{\R}{\mathbb{R}}
\newcommand{\cH}{\mathcal H}
\newcommand{\dd}{\,\mathrm d}
\newcommand{\D}{\mathrm D}
\newcommand{\1}{\mathds{1}}

\newcommand{\myitem}[1]{\item\label{#1}}
\newcommand{\myref}[1]{\hyperref[#1]{\textup{\ref*{#1}}}}

\theoremstyle{plain}
\newtheorem{theorem}{Theorem}
\newtheorem{proposition}{Proposition}
\newtheorem{lemma}{Lemma}[section]

\theoremstyle{definition}
\newtheorem{assumption}{Assumption}

\newtheorem{remark}{Remark}

\newcounter{proofstep}

\DeclareMathOperator{\supp}{supp}

\title[Moments and boundaries in Volterra volatility models]
{Moments and Boundary Attainment in Volterra Volatility Models: Bergomi and Rough Heston}

\author{Arthur Bourdon}
\address{CERMICS, CNRS, ENPC, Institut Polytechnique de Paris, Marne-la-Vall\'ee, France}
\email{arthur.bourdon@enpc.fr}

\author{Thibault Jeannin}
\address{CERMICS, CNRS, ENPC, Institut Polytechnique de Paris, Marne-la-Vall\'ee, France}
\email{thibault.jeannin@enpc.fr}

\date{}

\begin{document}
\begin{abstract}
We study two probabilistic questions for stochastic Volterra equations arising in rough volatility. These equations underlie some of the most popular non-Markovian stochastic volatility models in mathematical finance. First, we establish subcritical positive moment bounds for stochastic exponentials driven by Gaussian Volterra processes. In the Gaussian Volterra--Bergomi setting, we prove that if $\rho\in[-1,0)$, then $\E[S_T^p]<\infty$ for every $0<p<p_\rho$, where $p_{-1}=\infty$ and $p_\rho=(1-\rho^2)^{-1}$ for $-1<\rho<0$. For the fractional rough Bergomi kernel, we additionally prove explosion at the critical exponent $p=p_\rho$. Combined with the known explosion above the threshold, this yields the exact criterion $\E[S_T^p]<\infty$ if and only if $0<p<p_\rho$ in the fractional rough Bergomi model. Second, for the fractional Volterra square-root process---equivalently, the rough Heston variance process---we prove that its law has a positive atom at zero at every positive time. In particular, no Feller-type condition can make the zero boundary inaccessible in the fractional rough Heston regime.
\end{abstract}
\maketitle

\section{Introduction and Main Results}

Rough volatility models were introduced to reproduce the low regularity observed in volatility time series. The empirical analysis of \citet{GatheralJaissonRosenbaum2018} indicates that log-volatility behaves, over a wide range of scales, like a fractional process with Hurst index significantly smaller than one half. This led to non-Markovian stochastic volatility models in which the volatility factor is a Volterra transform of Brownian motion. Two benchmark models are the rough Bergomi model of \citet{BayerFrizGatheral2016}, which is lognormal and particularly useful for smile modelling and simulation, and the rough Heston model of \citet{ElEuchRosenbaum2019}, which belongs to the affine Volterra class of \citet{AbiJaberLarssonPulido2019}.

\medskip

The first question considered here concerns moment finiteness in the rough Bergomi model. The stock price process is a stochastic exponential. It is therefore a non-negative local martingale and a supermartingale, but its true martingale property and its higher moments are delicate because the volatility is lognormal and non-Markovian. \citet{Gassiat2019} proved that, in the rough Bergomi model, the price is a true martingale if and only if the correlation is non-positive, and that, for $-1<\rho<0$, all moments of order $p>(1-\rho^2)^{-1}$ are infinite at every positive maturity under an Osgood growth condition, which is automatic for the lognormal response. We state the corresponding kernel-dependent criterion directly under Assumption \ref{ass:kernel} in Lemma \ref{lem:osgood-supercritical-explosion} of Appendix \ref{subsec:app-osgood-supercritical}; the discussion following the lemma explains why this is an immediate adaptation of Gassiat's Riemann--Liouville argument. We prove finiteness throughout the strict subcritical range and, for the fractional rough Bergomi kernel, explosion at the critical exponent $p=(1-\rho^2)^{-1}$. Thus the fractional model satisfies the exact criterion $\E[S_T^p]<\infty$ if and only if $0<p<(1-\rho^2)^{-1}$. In the classical Markovian lognormal stochastic-volatility setting, martingality and moment phenomena of the same nature were studied by \citet{Sin1998,Jourdain2004,LionsMusiela2007}; in particular, the threshold $(1-\rho^2)^{-1}$ already appears there. Critical moment indices also enter the model-independent wing formula of \citet{Lee2004}, which makes subcritical moment bounds relevant for implied-volatility extrapolation. Moment control is also useful for Monte Carlo central-limit-theorem estimates and for the convergence analysis of discretization schemes. \citet{GerholdPachschwoellRuf2024} recently studied integrability of the supremum of stochastic-volatility martingales, including rough Bergomi, as a way to obtain useful integrability results without relying on higher stock-price moments. The subcritical proof replaces the classical
It\^o--Lyapunov argument by a deterministic Volterra chain-rule estimate, then
transfers the resulting Cameron--Martin bound to the Brownian Volterra input by
Borell's Gaussian isoperimetric inequality. The critical explosion proof instead uses finite-dimensional Gaussian conditioning: the singular fractional diagonal creates a conditional It\^o correction of order $\Delta^{H-1/2}$, which diverges as the conditioning mesh $\Delta$ tends to zero.

\medskip

The second question concerns the zero boundary in rough Heston. In the classical CIR and Heston models, the Feller condition determines whether the square-root process remains strictly positive or can reach the degenerate boundary. In the Volterra setting, especially in the rough regime, the issue is more delicate: before the first hitting time of zero, the square-root coefficient is locally Lipschitz and pathwise arguments can be localized, but after hitting zero the diffusion coefficient degenerates and such arguments break down. Boundary attainment is therefore part of the well-posedness problem. It also matters for statistical results on Volterra--CIR models: for instance, \citet{BenAlayaFriesenKremer2024} derive maximum-likelihood estimators of the drift parameters under uniform negative-moment assumptions such as
$\sup_{t\ge0}\E[X_t^{-1-\varepsilon}]<\infty$. Such assumptions necessarily fail whenever a time marginal assigns positive mass to zero. For the fractional rough Heston variance process, the Volterra memory and the singular fractional kernel modify the boundary mechanism. The affine Volterra framework provides existence, uniqueness in law, and transform formulae; see \citet{AbiJaberLarssonPulido2019}, \citet{ElEuchRosenbaum2019}, and \citet{FriesenJin2024}. Moreover, \citet{FriesenJin2024} prove absolute continuity of the time marginals on the interior of the state space. In the one-dimensional Volterra square-root case, the law at a fixed time is therefore the sum of a possible atom at zero and an absolutely continuous component on $(0,\infty)$, but this does not determine whether the boundary atom has positive mass. The result below answers this question in the fractional rough Heston case: zero is hit with positive probability before every positive horizon.

Our result can be viewed as a first step towards extending to rough kernels the
work of \citet{BondiPulido2024}, who establish Feller-type non-exit criteria
for invariant domains in smooth-kernel Volterra equations. Although our proof
uses the affine structure of the rough Heston model, through the explicit
Laplace transform of the variance process, the heuristic rough Feller scale discussed below
suggests why boundary attainment should occur in the fractional rough regime. A related result was obtained independently and contemporaneously by \citet{FriesenGerholdWiedermann2026}, who study boundary behaviour of Volterra square-root processes. Our present proof is more direct because it is tailored to the fractional rough
Heston case, whereas their work treats more general kernels.

\subsection{Bergomi model and moment threshold}

The Gaussian Volterra--Bergomi model considered in this section is
\begin{equation}
\label{eq:model-S}
\forall t\in[0,T],\quad
\begin{cases}
S_t=S_0+\int_0^t S_s \sigma(s,Y_s)\dd W_s, \\
Y_t=y_0 + \nu\int_0^t K(t-s)\dd B_s, \\
W_t=\rho B_t+\sqrt{1-\rho^2}\,B_t^\perp,
\end{cases}
\end{equation}
where $S_0>0$, $y_0\in\R$, $\nu>0$, $\rho\in[-1,1]$, $B$ and
$B^\perp$ are independent one-dimensional Brownian motions, $K$ is a
deterministic kernel satisfying Assumption \ref{ass:kernel} below, and $T >0$ is a finite maturity. In the following, Volterra convolutions on $[0,T]$ will be denoted by
\[
\forall t\in[0,T],\quad
(f*g)(t)=\int_0^t f(t-s)g(s)\dd s,
\]
whenever the integral is well-defined. If $\mu$ is a finite measure on $[0,T]$, then
\[
\forall t\in[0,T],\quad
(\mu*f)(t)=\int_0^t f(t-s)\mu(\dd s).
\]
The model
is defined for all correlations $\rho\in[-1,1]$; the moment result below is
a negative-correlation result, stated for $\rho\in[-1,0)$. 
In particular, the rough Bergomi model with a deterministic forward variance
curve is recovered by taking a lognormal response
\[
\sigma(t,x)=m(t)e^x,
\quad
m\in C^1([0,T]),\quad \inf_{[0,T]}m>0.
\]
The analysis below also covers the Volterra Stein--Stein response
\[
\sigma(t,x)=\alpha x+\beta(t),
\quad
\alpha>0,
\quad
\beta\in C([0,T]),
\]
and the odd-polynomial, in particular quintic, Gaussian Volterra volatility
specifications
\[
\sigma(t,x)=P_5(x),
\quad
P_5(x)=a_5x^5+\sum_{j=0}^4a_jx^j,
\quad
a_5>0,
\]
used in
\citet{AbiJaber2020GaussianSV,AbiJaberIllandLi2022,AbiJaberHainautMotte2025}.
The usual Stein--Stein and quintic Ornstein--Uhlenbeck presentations may include
an additional linear Volterra mean-reversion term in the Gaussian factor:
\[
Y_t=y_0(t)+\int_0^tK(t-s)(\theta(s)-\lambda Y_s)\dd s
+\nu\int_0^tK(t-s)\dd B_s,
\quad \lambda\ge0.
\]
This term is harmless and can be absorbed into the kernel by the linear
resolvent reduction. Indeed, if $R_\lambda$ is the second-kind resolvent of
$\lambda K$ and $K_\lambda:=K-R_\lambda*K$, then
\[
Y_t
=
y_0(t)+(K*\theta)(t)
-\bigl(R_\lambda*(y_0+K*\theta)\bigr)(t)
+\nu\int_0^tK_\lambda(t-s)\dd B_s.
\]
Whenever the reduced kernel $K_\lambda$ satisfies Assumption~\ref{ass:kernel},
the moment argument applies without further change.
\medskip
\begin{assumption}\label{ass:kernel}
The kernel $K:(0,T]\to\R$ satisfies the following conditions.
\begin{enumerate}[label=(\roman*),ref=(\roman*)]
\myitem{item:kernel-l2} $K\ge0$ and $K\in L^2([0,T])$.
\myitem{item:kernel-modulus}
There exist constants $C_K<\infty$ and $\gamma_K>0$ such that
\begin{equation}\label{eq:kernel-increment-modulus}
\forall h\in[0,T],\quad
\int_0^h K(s)^2\dd s
+
\int_0^{T-h}|K(s+h)-K(s)|^2\dd s
\le
C_K h^{2\gamma_K}.
\end{equation}
\myitem{item:kernel-resolvent} There exists a finite positive measure $\mathcal R_K$ on $[0,T]$, of the form
\begin{equation}\label{eq:resolvent-form}
\mathcal R_K(\dd t)=q_0\delta_0(\dd t)+q_K(t)\dd t,
\end{equation}
where $q_0\ge0$, $q_K\in L^1([0,T])$, $q_K\ge0$, and $q_K$ is non-increasing, such that
\begin{equation}\label{eq:resolvent-identity}
\forall t\in(0,T],\quad
(\mathcal R_K*K)(t)=1.
\end{equation}
\end{enumerate}
\end{assumption}

\begin{remark}
\label{rem:kernel-modulus-role}
\begin{enumerate}
\item Assumption \ref{ass:kernel} is readily verified for the standard kernels used
below: smooth non-negative kernels with a positive first-kind resolvent, such as
the constant and positive finite multi-exponential kernels, and the Riemann--Liouville kernel.

\item Assumption \myref{item:kernel-modulus} is a simple sufficient form of
Dudley's canonical Gaussian regularity condition. Indeed, it implies Hölder
continuity of the canonical $L^2$-metric associated with the Volterra
kernel, and therefore the corresponding Dudley entropy integral is finite; see
\citet{Dudley1967}. In particular, it ensures that the Volterra Gaussian input
has continuous paths and that its Cameron--Martin space is embedded in
$C_0([0,T])$; see
\myref{item:volterra-deterministic-continuity},
\myref{item:volterra-gaussian-continuity}, and
\myref{item:volterra-cm-space} of Lemma \ref{lem:volterra-gaussian-cm}.
\end{enumerate}
\end{remark}

\begin{assumption}\label{ass:volatility}
The map $\sigma:[0,T]\times\R\to\R$ is continuous and locally Lipschitz in the space variable, uniformly in time on compact space intervals. Moreover, there exist
\[
m\in C^1([0,T];(0,\infty)),
\quad
\mathcal S\in C^2(\R;[0,\infty))
\]
such that $\mathcal S$ is convex. We write
\begin{equation}\label{eq:lyapunov-primitive}
s(x):=\mathcal S'(x),
\quad x\in\R.
\end{equation}
Then $s$ is non-decreasing, and we assume that, for some $C_s<\infty$,
\begin{equation}\label{eq:s-derivative-growth}
0\le s'(x)\le C_s(1+|s(x)|),
\quad x\in\R.
\end{equation}
We define the remainder by
\begin{equation}\label{eq:sigma-decomposition}
r_\sigma(t,x):=\sigma(t,x)-m(t)s(x),
\quad (t,x)\in[0,T]\times\R.
\end{equation}
Since $m$ is continuous and strictly positive on $[0,T]$, we set
\[
m_-:=\min_{t\in[0,T]}m(t)>0,
\quad
m_+:=\max_{t\in[0,T]}m(t)<\infty.
\]
The remainder is asymptotically negligible with respect to $s$, uniformly in time: for every $\delta>0$, there exists $C_\delta<\infty$ with
\begin{equation}\label{eq:r-negligible}
|r_\sigma(t,x)|\le \delta |s(x)|+C_\delta,
\quad (t,x)\in[0,T]\times\R.
\end{equation}
Finally, the time weight $m$ and the Lyapunov function $\mathcal S$ satisfy
\begin{equation}\label{eq:weight-lyapunov-compatibility}
\begin{aligned}
&\text{either $m$ is non-increasing,} \\
&\text{or for every $\delta>0$, there exists $\widetilde{C}_\delta<\infty$ such that}
\quad
\mathcal S(x)\le \delta s(x)^2+\widetilde{C}_\delta,
\quad x\in\R.
\end{aligned}
\end{equation}
\end{assumption}

\begin{remark}\label{rem:compact-volatility-assumption}
The two lower-order requirements \eqref{eq:r-negligible} and
\eqref{eq:weight-lyapunov-compatibility} admit the following equivalent compact
form. With $(m')_+(t):=\max\{m'(t),0\}$, for every $\delta>0$, there exists
$C_\delta<\infty$ such that
\begin{equation}\label{eq:compact-volatility-assumption}
\sup_{t\in[0,T]}
\left|\sigma(t,x)-m(t)\mathcal S'(x)\right|^2
+
\left\|(m')_+\right\|_{L^\infty(0,T)}\mathcal S(x)
\le
\delta\left|\mathcal S'(x)\right|^2+C_\delta,
\quad x\in\R.
\end{equation}
\end{remark}

We henceforth focus on the negative-correlation regime $\rho\in[-1,0)$, and define the critical exponent by
\begin{equation}\label{eq:critical-exponent}
p_\rho:=\begin{cases}
\infty, &\text{if } \rho=-1,\\[2mm]
\dfrac{1}{1-\rho^2}, &\text{if } \rho \in (-1,0).
\end{cases}
\end{equation}

The following theorem establishes the finite side of the moment threshold. Under the general kernel and volatility assumptions introduced above, every positive moment strictly below $p_\rho$ is finite. In particular, the result applies to the lognormal Volterra--Bergomi, Volterra Stein--Stein, and quintic Gaussian Volterra specifications described above.

\begin{theorem}\label{thm:bergomi-main}
Let $\rho\in[-1,0)$. Assumptions \ref{ass:kernel} and
\ref{ass:volatility} are in force. Then, for every $p\in(0,p_\rho)$, with
$p_\rho$ defined by \eqref{eq:critical-exponent},
$\E[S_T^p]<\infty.$
\end{theorem}

The following result characterizes the kernel-dependent behaviour of the asset-price moment at the critical exponent in the lognormal Volterra--Bergomi model with a deterministic forward variance curve. In particular, it identifies settings in which the critical moment is finite and others in which it is infinite.

\begin{theorem}
\label{prop:bergomi-critical-kernels}
Let $\rho\in(-1,0)$, and suppose that
\[
\sigma(t,x)=m(t)e^x,
\quad m\in C^1([0,T];(0,\infty)).
\]
Then the critical moment has the following kernel-dependent behaviour.
\begin{enumerate}[label=(\roman*),ref=(\roman*)]
\myitem{item:critical-constant-kernel}
\textbf{Constant kernel.} If $K\equiv1$, then
\[
\E\left[S_T^{p_\rho}\right]<\infty
\quad\Longleftrightarrow\quad
\frac{m'(t)}{m(t)}+\frac{\nu^2}{2}\le0
\quad\text{for every }t\in[0,T].
\]
In particular, if $m$ is constant, then
$\E[S_T^{p_\rho}]=\infty$.

\myitem{item:critical-multiexponential-kernel}
\textbf{Positive finite multi-exponential kernel.} If
\[
K(t)=\sum_{i=1}^N w_i e^{-\lambda_i t},
\quad
N\ge1,\quad w_i>0,\quad \lambda_i\ge0,
\quad \max_{1\le i\le N}\lambda_i>0,
\]
then
\[
\E\left[S_T^{p_\rho}\right]<\infty.
\]

\myitem{item:critical-rough-kernel}
\textbf{Fractional rough kernel.} If
\[
K(t)=K_H(t):=\frac{t^{H-\frac12}}{\Gamma(H+\frac12)},
\quad H\in\left(0,\frac12\right),
\]
then
\[
\E\left[S_T^{p_\rho}\right]=\infty.
\]
\end{enumerate}
Consequently, combined with Theorem \ref{thm:bergomi-main} and Lemma
\ref{lem:osgood-supercritical-explosion} of Appendix
\ref{subsec:app-osgood-supercritical}, whose Osgood condition is automatic
for the lognormal response under Assumption \ref{ass:kernel},
\[
\left\{p>0:\E[S_T^p]<\infty\right\}
=
\begin{cases}
(0,p_\rho],
& K\equiv1 \text{ and }
\frac{m'(t)}{m(t)}+\frac{\nu^2}{2}\le0
\text{ for every }t\in[0,T],\\
(0,p_\rho),
& K\equiv1 \text{ and }
\frac{m'(t)}{m(t)}+\frac{\nu^2}{2}>0
\text{ for some }t\in[0,T],\\
(0,p_\rho],
& K \text{ is as in \myref{item:critical-multiexponential-kernel}},\\
(0,p_\rho),
& K=K_H,\ H\in(0,\tfrac12).
\end{cases}
\]
\end{theorem}

\begin{remark}
\label{rem:explosion-critical-threshold}
\begin{enumerate}
\item Moment explosion above the threshold $p_\rho$ was established by
\citet{Gassiat2019} and is recalled for more general kernels in Lemma
\ref{lem:osgood-supercritical-explosion}.

\item For the quintic Gaussian Volterra volatility model with exponential
kernel
\[
K(t)=e^{-\lambda t},
\quad \lambda>0,
\]
a classical elementary It\^o argument shows that
$\E[S_T^{p_\rho}]<\infty.$

\item The equality case $p=p_\rho$ is genuinely kernel-sensitive.
Theorem \ref{prop:bergomi-critical-kernels} gives the exact comparison
for three benchmark lognormal kernel classes. The singular fractional kernel
\[
K_H(t)=\frac{t^{H-\frac12}}{\Gamma(H+\frac12)},
\quad H\in\left(0,\frac12\right),
\]
falls on the explosive side:
$\E[S_T^{p_\rho}]=\infty.$ The mechanism driving moment explosion in the rough case differs from that associated with a constant kernel. It arises from the divergent conditional Itô correction induced by the diagonal singularity $K_H(0+)=\infty.$
\end{enumerate}
\end{remark}

\begin{table}[H]
\centering
\caption{Critical and supercritical moments for the main model classes.}
\label{tab:critical-supercritical-moments}
{\footnotesize
\renewcommand{\arraystretch}{1.15}
\setlength{\tabcolsep}{3pt}
\begin{tabularx}{\textwidth}{
|>{\RaggedRight\arraybackslash}p{0.11\textwidth}
|>{\RaggedRight\arraybackslash}p{0.11\textwidth}
|>{\RaggedRight\arraybackslash}p{0.16\textwidth}
|>{\RaggedRight\arraybackslash}p{0.10\textwidth}
|>{\RaggedRight\arraybackslash}p{0.09\textwidth}
|>{\hsize=.75\hsize\linewidth=\hsize\RaggedRight\arraybackslash}X
|>{\hsize=1.25\hsize\linewidth=\hsize\RaggedRight\arraybackslash}X|}
\hline
Model
& $\sigma(t,x)$
& $m,s,\mathcal S,r_\sigma$
& Kernel
& Structural condition
& Supercritical regime $p>p_\rho$
& Critical moment $\E[S_T^{p_\rho}]$ \\
\hline

lognormal Volterra--Bergomi with forward variance
& $m(t)e^x$, $m\in C^1$, $m>0$
& $s=e^x$, $\mathcal S=e^x$, $r_\sigma=0$
& Convolution kernel $K$
& $\mathcal S\le\delta s^2+C_\delta$
& Infinite for every $p>p_\rho$; see Lemma
\ref{lem:osgood-supercritical-explosion}.
& For $K\equiv1$, finite exactly when
$m'/m+\nu^2/2\le0$ on $[0,T]$. Finite for every non-constant
positive finite multi-exponential kernel and infinite for $K=K_H$.
See Theorem \ref{prop:bergomi-critical-kernels}. \\
\hline

Volterra Stein--Stein, reduced form
& $\alpha x+\beta(t)$, $\alpha>0$, $\beta\in C([0,T])$
& $m\equiv\alpha$, $s=x$, $\mathcal S=x^2/2$, $r_\sigma=\beta$
& Convolution kernel $K$
& $m$ constant
& Not determined here; Lemma
\ref{lem:osgood-supercritical-explosion} does not apply.
& Not determined here. \\
\hline

quintic Gaussian Volterra volatility
& $P_5(x)$, leading coefficient $>0$
& $m\equiv a_5$, $s=x^5$, $\mathcal S=x^6/6$,
$r_\sigma=P_5-a_5x^5$
& Exponential kernel $K(t)=e^{-\lambda t}$, $\lambda>0$
& $m$ constant
& Infinite for every $p>p_\rho$; see Lemma
\ref{lem:osgood-supercritical-explosion}.
& Finite. \\
\hline
\end{tabularx}
}
\end{table}

\subsection{Rough Heston boundary} 
Let
\[
v_0>0,\quad b\geq 0,\quad \kappa\geq 0,\quad \nu>0.
\]
For $H\in(0,\frac12)$, we set
\[
K_H(t):=\frac{t^{H-\frac12}}{\Gamma(H+\frac12)},
\quad
\mathcal R_H(\dd t):=\frac{t^{-H-\frac12}}{\Gamma(\frac12-H)}\dd t,
\]
so that $\mathcal R_H*K_H=1$. The rough Heston variance process is the scalar Volterra CIR process
\begin{equation}\label{eq:VCIR}
V_t=v_0+\int_0^tK_H(t-s)(b-\kappa V_s)\dd s+\nu\int_0^tK_H(t-s)\sqrt{V_s}\dd W_s.
\end{equation}
The existence of a continuous non-negative weak solution and uniqueness in law are standard results for affine Volterra processes; see \citet{AbiJaberLarssonPulido2019}.

In the classical CIR model, accessibility of the zero boundary is governed by
the Feller condition. The following theorem shows that this mechanism breaks
down in the fractional rough regime: for every positive maturity, the law of
the variance process has a strictly positive atom at zero. Moreover, we obtain
an explicit lower bound for the mass of this atom.

\begin{theorem}
\label{thm:heston-main}
For every $T>0$, one has
\begin{equation}\label{eq:lower-bound-main}
\Prob(V_T=0)
\geq
\exp\left(
-v_0
\frac{4H\Gamma(\frac12-H)^2}{\nu^2\Gamma(1-2H)^2}
T^{-2H}
-
b
\frac{8H\Gamma(\frac12-H)}{\nu^2(1-2H)\Gamma(1-2H)}
T^{\frac12-H}
\right)
>0.
\end{equation}
\end{theorem}
In particular, since $V$ has continuous non-negative paths,
\[
    \Prob\left(\inf_{0\leq t\leq T}V_t=0\right)
    \geq \Prob(V_T=0)>0.
\]
Hence, the zero boundary is attained by every positive horizon with positive
probability, regardless of the values of $b$ and $\kappa$. Consequently, no
analogue of the classical Feller condition can make the zero boundary
inaccessible in the fractional rough Heston regime.

\section{Control of Volterra--Bergomi moments: Proof of Theorem \ref{thm:bergomi-main}}
\label{sec:thm1}

For $p>1$ and $\rho<0$, we first introduce the tilted stochastic Volterra
equation
\begin{equation}\label{eq:tilted-feedback}
\forall t\in[0,T],\quad
\widetilde Y_t
=
y_0+\nu\int_0^tK(t-s)\dd B_s
+
\nu p\rho\int_0^tK(t-s)\sigma(s,\widetilde Y_s)\dd s .
\end{equation}
Local existence
and pathwise uniqueness for \eqref{eq:tilted-feedback} follow from the standard
local theory for nonlinear Volterra equations with $L^1$-kernels and locally
Lipschitz nonlinearities; see \citet[Chapter 12]{GripenbergLondenStaffans1990},
and also \citet[Appendix B]{AbiJaberLarssonPulido2019} for the same convolution
setting. The global existence result needed below is established in
\myref{item:feedback-well-posedness} of Lemma \ref{lem:well-posed-general}.
The proof of Theorem \ref{thm:bergomi-main} relies on the following two
auxiliary results, whose proofs are given after the proof of Theorem
\ref{thm:bergomi-main}.

\medskip


\begin{lemma}\label{lem:girsanov}
Under Assumptions \ref{ass:kernel} and \ref{ass:volatility}, for $p>1$ and $\rho<0$, if the
solution $\widetilde Y$ of \eqref{eq:tilted-feedback} satisfies
\begin{equation}\label{eq:needed-integrability}
\E\left[
\exp\left(\frac12p(p-1)\int_0^T \sigma(t,\widetilde Y_t)^2\dd t\right)
\right]<\infty,
\end{equation}
then $\E[S_T^p]<\infty$.
\end{lemma}

\begin{proposition}\label{prop:negative-feedback}
Under Assumptions \ref{ass:kernel} and \ref{ass:volatility}, for $a>0$, the solution of
\begin{equation}\label{eq:negative-feedback}
\forall t\in[0,T],\quad
Y_t
=
y_0+\nu\int_0^tK(t-s)\dd B_s
-
\nu a\int_0^tK(t-s)\sigma(s,Y_s)\dd s
\end{equation}
satisfies
\begin{equation}\label{eq:exp-integrability}
\forall \Lambda \in [0,a^2/2),
\quad
\E\left[
\exp\left(\Lambda\int_0^T \sigma(t,Y_t)^2\dd t\right)
\right]
<\infty.
\end{equation}
\end{proposition}

\begin{remark}\label{rem:smooth-kernel-ito}
The main technical difficulty comes from the lack of smoothness of the kernel.
Let us explain what would happen in the smoother semimartingale regime. For this
remark only, suppose that $K\in C^1([0,T])$ and that $K_0:=K(0)>0$. Then
the Volterra process is a semimartingale. More precisely, if we write
$q(t):=\sigma(t,Y_t)$, then
\[
\dd Y_t
=
\nu K_0\dd B_t
-\nu aK_0q(t)\dd t
+\beta_t\dd t,
\]
where
\[
\beta_t
=
\nu\int_0^tK'(t-s)\dd B_s
-
\nu a\int_0^tK'(t-s)q(s)\dd s .
\]
Thus the negative feedback produces the instantaneous drift
$-\nu aK_0\sigma(t,Y_t)\dd t.$
Applying It\^o's formula to the Lyapunov function $\mathcal S$, with
$\mathcal S'=s$, gives, after localization,
\[
\dd\mathcal S(Y_t)
=
\nu K_0s(Y_t)\dd B_t
-\nu aK_0s(Y_t)\sigma(t,Y_t)\dd t
+s(Y_t)\beta_t\dd t
+\frac12\nu^2K_0^2s'(Y_t)\dd t .
\]
Since
\[
\sigma(t,x)=m(t)s(x)+r_\sigma(t,x),
\quad
m(t)\ge m_->0,
\]
the leading drift term satisfies, up to lower-order contributions,
\[
-\nu aK_0s(Y_t)\sigma(t,Y_t)
=
-\nu aK_0m(t)s(Y_t)^2
-\nu aK_0s(Y_t)r_\sigma(t,Y_t).
\]
The first term is coercive. The second term is absorbed by the negligibility
assumption on $r_\sigma$. The Gaussian and feedback components of $\beta_t$
satisfy, by the boundedness of $K'$ on $[0,T]$,
\[
\E\left[
\int_0^t
\left|
\nu\int_0^sK'(s-u)\dd B_u
\right|^2
\dd s
\right]
\le C_T.
\]
Moreover,
\[
\E\left[
\int_0^t
\left|
\nu a\int_0^sK'(s-u)q(u)\dd u
\right|^2
\dd s
\right]
\le
C_T\int_0^t
\E\left[\int_0^s q(u)^2\dd u\right]\dd s.
\]
Young's inequality then allows the $s(Y_s)^2$-part to be absorbed by the
coercive drift. This yields an estimate of
the form
\[
\E[\mathcal S(Y_t)]
+
c\E\left[\int_0^t\sigma(s,Y_s)^2\dd s\right]
\le
C
+
C\int_0^t
\E\left[\int_0^s\sigma(u,Y_u)^2\dd u\right]\dd s .
\]
Grönwall's lemma then gives
\[
\E\left[\int_0^T\sigma(t,Y_t)^2\dd t\right]<\infty .
\]
This is the type of It\^o--Lyapunov argument used in the Markovian lognormal
semimartingale setting, for instance in \citet{Jourdain2004}. Written in
Volterra form, this corresponds to the exponential-kernel case
$K(t)=e^{-\lambda t}$. For rough Volterra kernels, the process $Y$ is no
longer a semimartingale and the above It\^o argument is not available. The proof below replaces the instantaneous It\^o coercivity by a
deterministic Volterra chain-rule inequality, and then transfers the resulting
Cameron--Martin estimate to the Gaussian input through Borell's isoperimetric
inequality.
\end{remark}

We now provide the proof of Theorem \ref{thm:bergomi-main}.

\begin{proof}[Proof of Theorem \ref{thm:bergomi-main}]
For $p\in (0,1]$, the non-negative local martingale $S$ is a supermartingale,
and $x^p\le1+x$ on $[0,\infty)$. Thus
$\E[S_T^p]\le1+S_0<\infty.$
It remains to consider $p>1$.
Since $\rho<0$, one has $-p\rho>0$. The
condition $p(1-\rho^2)<1$ is equivalent to
$\frac12p(p-1)<\frac12p^2\rho^2.$
Proposition \ref{prop:negative-feedback}, applied with
$a=-p\rho$ and $\Lambda=\frac12p(p-1)$, gives
\eqref{eq:needed-integrability}. Lemma \ref{lem:girsanov} then yields
$\E[S_T^p]<\infty$.
\end{proof}

We now provide the proof of Lemma \ref{lem:girsanov}.

\begin{proof}[Proof of Lemma \ref{lem:girsanov}]
For $R>0$, we set
$\tau_R:=\inf\{t\in[0,T]:|Y_t|\ge R\}\wedge T.$ On $[0,\tau_R]$, the process $t\mapsto\sigma(t,Y_t)$ is bounded. Since
$W=\rho B+\sqrt{1-\rho^2}B^\perp$, conditioning on $B$ and using the
independence of $B^\perp$ gives
\[
\E[S_{\tau_R}^p]
=
S_0^p\E\left[
\exp\left(
 p\rho\int_0^{\tau_R} \sigma(t,Y_t)\dd B_t
+
\left(\frac12p^2(1-\rho^2)-\frac12p\right)
\int_0^{\tau_R}\sigma(t,Y_t)^2\dd t
\right)
\right].
\]
By Novikov's criterion, the stochastic exponential $\mathcal E^{p,R}$ defined by
\[ 
\forall t \in [0,T],\quad 
\mathcal E_t^{p,R}
=
\exp\left(
 p\rho\int_0^{\tau_R \wedge t} \sigma(s,Y_s)\dd B_s
-
\frac12p^2\rho^2\int_0^{\tau_R \wedge t} \sigma(s,Y_s)^2\dd s\right), 
\]
is a true martingale. Therefore, under the probability $\Q^{p,R}$ defined by
$\frac{\dd\Q^{p,R}}{\dd\Prob}=\mathcal E_T^{p,R}$, the process
\[
\left(B_t^{p,R}
:=
B_t-p\rho\int_0^{t\wedge\tau_R}\sigma(s,Y_s)\dd s\right)_{t\in[0,T]},
\]
is a Brownian motion under $\Q^{p,R}$ by Girsanov's theorem. Combining the
preceding identity with the
change of measure gives
\begin{equation}\label{eq:stopped-moment-identity}
\E[S_{\tau_R}^p]
=
S_0^p\E^{\Q^{p,R}}
\left[
\exp\left(\frac12p(p-1)\int_0^{\tau_R} \sigma(t,Y_t)^2\dd t\right)
\right].
\end{equation}
Under $\Q^{p,R}$, the volatility factor satisfies
\begin{equation}\label{eq:stopped-feedback}
\forall t\in[0,T],\quad
Y_t
=
y_0+\nu\int_0^tK(t-s)\dd B_s^{p,R}
+
\nu p\rho\int_0^tK(t-s)\sigma(s,Y_s)\1_{\{s\le\tau_R\}}\dd s.
\end{equation}
This is the stopped version of the tilted equation
\eqref{eq:tilted-feedback}. We also set $\widetilde\tau_R:=\inf\{t\in[0,T]:|\widetilde Y_t|\ge R\}\wedge T$. By pathwise uniqueness for the stopped Volterra equation, the stopped path
$(Y_{t\wedge\tau_R})_{t\in[0,T]}$ under $\Q^{p,R}$ has the same law as
$(\widetilde Y_{t\wedge\widetilde\tau_R})_{t\in[0,T]}$. It follows that
\[
\E[S_{\tau_R}^p]
=
S_0^p\E
\left[
\exp\left(
\frac12p(p-1)\int_0^{\widetilde\tau_R}\sigma(t,\widetilde Y_t)^2\dd t
\right)
\right].
\]
Condition \eqref{eq:needed-integrability} yields a constant
$C<\infty$, independent of $R$, such that
$\E[S_{\tau_R}^p]\le C$. Since $\tau_R\uparrow T$ almost surely and
$S$ has continuous paths, Fatou's lemma gives
\[
\E[S_T^p]
\le
\liminf_{R\to\infty}\E[S_{\tau_R}^p]
\le
C.
\]
\end{proof}

\subsection{Proof of Proposition \ref{prop:negative-feedback}}

The proof of Proposition \ref{prop:negative-feedback} relies on the following lemmas. Their proofs are postponed in Appendix \ref{subsec:app-volterra-gaussian-cm}.

\begin{lemma}
\label{lem:volterra-gaussian-cm}
Under Assumption \ref{ass:kernel}, the following facts hold.
\begin{enumerate}[label=(\roman*),ref=(\roman*)]
\myitem{item:volterra-deterministic-continuity}
For every $f\in L^2([0,T])$, the convolution $K*f$ has a continuous
representative on $[0,T]$, vanishing at zero.

\myitem{item:volterra-gaussian-continuity}
The Gaussian Volterra convolution $G:=\nu K * \dd B$ defined as
\begin{equation}\label{eq:G-def}
\forall t\in[0,T],\quad
G_t=\nu\int_0^t K(t-s)\dd B_s
\end{equation}
has a continuous version on $[0,T]$, with $G_0=0$.

\myitem{item:volterra-cm-space}
Let $\mu$ be the law of $G$ on $C([0,T])$. The Cameron--Martin space of the Gaussian measure $\mu$, denoted by $\cH_G$, is
\begin{equation}\label{eq:cm-space}
\cH_G
=
\{\nu K*f:f\in L^2([0,T])\},
\quad
\|\nu K*f\|_{\cH_G}
=
\|f\|_{L^2([0,T])}.
\end{equation}
Moreover,
\begin{equation}\label{eq:support-C0}
\supp\mu=C_0([0,T]),
\quad
C_0([0,T]):=\{x\in C([0,T]):x_0=0\}.
\end{equation}
\end{enumerate}
\end{lemma}

\begin{lemma}
\label{lem:borell}
Let $G=\nu K*\dd B$, and let $\mu$ be its law on $C_0([0,T])$. Let
$\Psi:C_0([0,T])\to[0,\infty]$ be Borel. Assume that there exist
$\alpha>0$, $B>0$, and $C<\infty$ such that
\[
\forall g\in C_0([0,T]) \text{ with } \|g\|_\infty\le\alpha,
\quad
\forall f\in L^2([0,T]),\quad
\Psi(g+\nu K*f)
\le
B\|f\|_{L^2([0,T])}^2+C.
\]
Then, for every $\delta\in(0,1/2)$, there exists
$C_\delta<\infty$ such that
\[
\forall R\ge0,\quad
\mu(\Psi>R)
\le
C_\delta
\exp\left(-\left(\frac12-\delta\right)\frac{R}{B}\right).
\]
\end{lemma}

\begin{remark}
The lemma is the infinite-dimensional analogue of the following elementary
one-dimensional fact. Let $X\sim N(0,1)$ and $\Psi : \mathbb R \to [0,\infty]$. If
$\Psi(a+h)\le Bh^2+C$ for every $|a|\le\alpha$, then
$\{\Psi\le Br^2+C\}$ contains
$[-\alpha,\alpha]+[-r,r],$
and hence
$\mathbb{P}(\Psi(X)>Br^2+C)
\le
\mathbb{P}(|X|>\alpha+r),$
which has Gaussian decay. In the lemma, $X$ is replaced by the Gaussian
Volterra process $G$, the interval $[-\alpha,\alpha]$ by
$A_\alpha=\{g \in C_0([0,T]) ~|~\|g\|_\infty\le\alpha\}$, the shifts $[-r,r]$ by the
Cameron--Martin ball $r\mathbb B_{\cH_G}:=\{ g  \in  \cH_G ~|~ \|g\|_{\cH_G} \le r\}$, and the estimate
$\Psi(a+h)\le Bh^2+C$ by
$\Psi(g+\nu K*f)\le B\|f\|_{L^2([0,T])}^2+C.$
Borell's inequality is the corresponding Gaussian tail estimate for
Cameron--Martin enlargements.
\end{remark}

\begin{lemma}\label{lem:well-posed-general}
Assumptions \ref{ass:kernel} and \ref{ass:volatility} are in force. Let
$a>0$. The following assertions hold.
\begin{enumerate}[label=(\roman*),ref=(\roman*)]
\myitem{item:feedback-well-posedness}
For every $x\in C([0,T])$, the Volterra equation
\begin{equation}\label{eq:pathwise-volterra}
\forall t\in[0,T],\quad
Y_t^x
=
y_0+x_t-\nu a\int_0^tK(t-s)\sigma(s,Y_s^x)\dd s
\end{equation}
has a unique continuous solution in $C([0,T])$.

\myitem{item:feedback-continuity}
The solution map
$x\longmapsto Y^x$
is continuous from $C([0,T])$ to itself.

\myitem{item:feedback-coercivity}
Define
\[
\Psi:C([0,T])\to[0,\infty],
\quad
x \mapsto \int_0^T\sigma(t,Y_t^x)^2\dd t.
\]
For every $B>a^{-2}$, there exist $\alpha>0$ and
$C<\infty$ such that
\[
\forall g\in C_0([0,T]) \text{ with } \|g\|_\infty\le\alpha,
\quad
\forall f\in L^2([0,T]),\quad
\Psi(g+\nu K*f)
\le
B\|f\|_{L^2([0,T])}^2+C.
\]
\end{enumerate}
\end{lemma}

We now provide the proof of Proposition \ref{prop:negative-feedback}.

\begin{proof}[Proof of Proposition \ref{prop:negative-feedback}]
We fix $\Lambda\in[0,a^2/2)$. If $\Lambda=0$, the result holds trivially. We
assume $\Lambda>0$. By \myref{item:volterra-gaussian-continuity} of Lemma
\ref{lem:volterra-gaussian-cm}, the process $G$ defined in
\eqref{eq:G-def} is a $C([0,T])$-valued Gaussian random variable. Since
$G_0=0$, we denote by $\mu$ its law on $C_0([0,T])$. By
\myref{item:feedback-well-posedness} of Lemma
\ref{lem:well-posed-general}, for every $x\in C([0,T])$, there exists a
unique solution $Y^x\in C([0,T])$ to
\[
\forall t\in[0,T],\quad
Y_t^x
=
y_0+x_t-
\nu a\int_0^tK(t-s)\sigma(s,Y_s^x)\dd s.
\]
We denote by $\Psi$ the map defined in \myref{item:feedback-coercivity} of
Lemma \ref{lem:well-posed-general}. By \myref{item:feedback-continuity} of
Lemma \ref{lem:well-posed-general}, $\Psi$ is Borel measurable. We still
denote by $\Psi$ its restriction to $C_0([0,T])$. Since the solution of
\eqref{eq:negative-feedback} is exactly $Y^G$, we have
\[
\int_0^T \sigma(t,Y_t)^2\dd t
=
\Psi(G).
\]
Since $\Lambda<a^2/2$, we choose $B>a^{-2}$ and $\delta\in(0,1/2)$ such that
$\Lambda
<
\frac{1/2-\delta}{B}.$
We take $\alpha,C$ from \myref{item:feedback-coercivity} of Lemma
\ref{lem:well-posed-general}. Then, for every $g\in C_0([0,T])$ with
$\|g\|_\infty\le\alpha$ and every $f\in L^2([0,T])$,
\[
\Psi(g+\nu K*f)
\le
B\|f\|_{L^2([0,T])}^2+C.
\]
Therefore, Lemma \ref{lem:borell} applies with this $\alpha$, the above
$B$, and this $C$. It yields
\begin{equation}\label{eq:psi-tail}
\forall R\ge0,\quad
\mu(\Psi>R)
\le
C_\delta
\exp\left(-\left(\frac12-\delta\right)\frac{R}{B}\right).
\end{equation}
Finally, the layer-cake formula and a change of variables give
\[
\E\left[
\exp\left(\Lambda\int_0^T \sigma(t,Y_t)^2\dd t\right)
\right]
=
\E_\mu\left[e^{\Lambda\Psi}\right] = \int_{0}^{\infty} \mu(e^{\Lambda \Psi} > R) dR =
1+\Lambda\int_0^\infty e^{\Lambda R}\mu(\Psi>R)\dd R.
\]
By \eqref{eq:psi-tail}, it follows that
\[
\E\left[
\exp\left(\Lambda\int_0^T \sigma(t,Y_t)^2\dd t\right)
\right]
\le
1+\Lambda C_\delta
\int_0^\infty
\exp\left(
-\left(
\frac{1/2-\delta}{B}
-\Lambda
\right)R
\right)\dd R
<\infty,
\]
by the choice of $B$ and $\delta$.
\end{proof}

\section{Critical moments in the lognormal Volterra--Bergomi model: Proof of Theorem \ref{prop:bergomi-critical-kernels}}

Assertion \myref{item:critical-constant-kernel} and the single-exponential case
of \myref{item:critical-multiexponential-kernel} follow from the classical
arguments in \citet[Proposition~6]{Jourdain2004} and
\citet[Theorem~2.3 and Section~2.5]{LionsMusiela2007}; the deterministic scale
$m$ only contributes the drift term $m'/m$. We record the short extension to
positive finite sums.

\begin{proof}[Proof of Theorem \ref{prop:bergomi-critical-kernels}]
For \myref{item:critical-multiexponential-kernel}, set $k_0:=K(0)$. After
combining equal rates, a partial-fraction decomposition of
$1/(z\widehat K(z))$ shows that the first-kind resolvent has the form
\[
\mathcal R_K(\dd t)=k_0^{-1}\delta_0(\dd t)+r(t)\dd t,
\]
where $r$ is non-increasing and $r(T)>0$. Set
\[
\forall t \in [0,T], \quad X_t:=Y_t-y_0,
\quad q_t:=m(t)e^{Y_t},
\quad M_T:=\int_0^Tq_t\dd B_t,
\quad
Z_t:=r(0)X_t+\int_0^t r'(t-s)X_s\dd s.
\]
Since $\mathcal R_K*X=\nu B$, It\^o's formula gives
\[
\nu M_T
=
\frac{q_T-q_0}{k_0}
+
\int_0^Tq_t\left(
Z_t-\frac1{k_0}\left(\frac{m'(t)}{m(t)}+\frac{\nu^2k_0^2}{2}\right)
\right)\dd t.
\]
Moreover,
\[
Z_t=r(t)X_t+\int_0^t(X_t-X_{t-u})(-r'(u))\dd u.
\]
Using $e^x(x-y)\ge e^x-e^y$ and $\inf_{t\in[0,T]}m(t)r(t)>0$, the integral on the
right is bounded below by
\[
e^{y_0}\int_0^T\bigl(m(t)r(t)X_t-C\bigr)e^{X_t}\dd t\ge-C_T.
\]
Hence $M_T\ge-C_T$, and conditioning on $B$ at the critical exponent yields
$\E[S_T^{p_\rho}]/S_0^{p_\rho}=\E[e^{-aM_T}]<\infty$, where
$a=-\rho/(1-\rho^2)>0$. This proves
\myref{item:critical-multiexponential-kernel}.

We now prove \myref{item:critical-rough-kernel}. Throughout the remainder of
the proof, we set $\alpha:=H+\frac12\in\left(\frac12,1\right)$ and
\[
\forall t \in [0,T], \quad 
q_t:=m(t)e^{Y_t},
\quad
M_T:=\int_0^Tq_t\dd B_t, \quad
K_H(t)=\frac{t^{\alpha-1}}{\Gamma(\alpha)}.
\]
At the critical exponent, conditioning on the Brownian motion $B$ removes the
$\int_0^Tq_t^2\dd t$-term from the moment formula. Thus it is enough to prove that
$\E[e^{-aM_T}]=\infty$ for $a>0$. We condition further on the normalized
increments of $B$ over a fine mesh. The resulting conditional mean of $M_T$
contains two competing terms:
\[
\text{a control term of size }R e^{\nu R}
\quad\text{and}\quad
\text{a negative diagonal term of size }n^{1-\alpha}e^{\nu R}.
\]
The second term is specific to the singular fractional kernel. Choosing the
mesh so that $n^{1-\alpha}\gg R$ makes the conditional mean very negative on
a suitable Gaussian event. Conditional Jensen's inequality then turns this
into the desired explosion.

\medskip
\noindent\textbf{Step 1: reduction to a negative exponential moment.}
Conditionally on $B$, the process $q$ is independent of $B^\perp$. Hence,
for every $p>0$, the explicit stochastic-exponential representation of $S$
gives, with equality in $[0,\infty]$,
\begin{equation}
\label{eq:critical-conditioning-identity}
\frac{\E[S_T^p]}{S_0^p}
=
\E\left[
\exp\left(
 p\rho M_T
 +\frac12\bigl(p^2(1-\rho^2)-p\bigr)\int_0^Tq_t^2\dd t
\right)
\right].
\end{equation}
At $p=p_\rho$, the coefficient of $\int_0^Tq_t^2\dd t$ vanishes. Therefore,
it remains to prove
\begin{equation}
\label{eq:negative-laplace-explosion}
\E[e^{-aM_T}]=\infty, \text{ with } a=-\frac{\rho}{1-\rho^2}>0.
\end{equation}
In fact, the argument below proves \eqref{eq:negative-laplace-explosion} for
every $a>0$.

\medskip

\noindent\textbf{Step 2: conditioning reveals the rough diagonal term.}
For $n\ge1$, we set $\Delta_n:=T/n$ and
\[
I_j^n:=((j-1)\Delta_n,j\Delta_n],
\quad
e_j^n:=\Delta_n^{-1/2}\1_{I_j^n},
\quad
Z_j^n:=\int_0^Te_j^n(s)\dd B_s
=\frac{B_{j\Delta_n}-B_{(j-1)\Delta_n}}{\sqrt{\Delta_n}},
\]
for $j=1,\ldots,n$. The vector $Z^n=(Z_1^n,\ldots,Z_n^n)$ is standard
Gaussian in $\R^n$. For $z\in\R^n$, we define
\[
h_z^n:=\sum_{j=1}^nz_je_j^n,
\quad
g_z^n:=K_H*h_z^n.
\]
We set
\[
\forall t \in [0,T], \quad
r_n(t):=\int_0^tK_H(t-s)^2\dd s-
\sum_{j=1}^nk_j^n(t)^2,
\quad
\kappa_n(t):=\sum_{j=1}^ne_j^n(t)k_j^n(t), \quad k_j^n(t):=\int_0^tK_H(t-s)e_j^n(s)\dd s.
\]
Here $g_z^n(t)$ and $r_n(t)$ are respectively the conditional mean and
variance of $\int_0^tK_H(t-s)\dd B_s$ given $Z^n=z$. We set
\[\forall t \in [0,T], \quad
W_n(t,z):=
m(t)\exp\left(
y_0+\nu g_z^n(t)+\frac{\nu^2}{2}r_n(t)
\right)
=\E[q_t\mid Z^n=z].
\]
Since $0 \le r_n(t) \le \int_{0}^{t} K_H(t-s)^2 \dd s$ for all $t \in [0,T]$, and $t \mapsto \int_{0}^{t} K_H(t-s)^2 \dd s$ is continuous, there are constants $0<c_-\le c_+<\infty$, independent of
$n$, such that
\begin{equation}
\label{eq:conditional-weight-bounds}
\forall (t,z)\in[0,T]\times\R^n, \quad c_-e^{\nu g_z^n(t)}
\le W_n(t,z)\le
c_+e^{\nu g_z^n(t)}.
\end{equation}
The conditional mean of $M_T$ is
\begin{equation}
\label{eq:conditional-mean-critical}
\E[M_T\mid Z^n=z]
=
\int_0^T
W_n(t,z)
\bigl(h_z^n(t)-\nu\kappa_n(t)\bigr)
\dd t.
\end{equation}
Indeed, let $F\in C_c^\infty(\R^n)$. Since $\E[q_t\mid Z^n]=W_n(t,Z^n)$ and $D_t F(Z^n) = \sum_{j=1}^{n} \partial_{j} F(Z^n) e_j^{n}(t)$ in $L^2([0,T] \times \Omega)$, where $D$ denotes the Malliavin derivative operator, Malliavin duality for the adapted integrand $q$
gives
\[
\E[F(Z^n)M_T] = \E\left[\int_{0}^{T}D_t F(Z^n) q_t \dd t\right] =\sum_{j=1}^n
\E\left[
\partial_jF(Z^n)
\int_0^Te_j^n(t)W_n(t,Z^n)\dd t
\right].
\]
Moreover, since $\partial_jW_n(t,z)=\nu k_j^n(t)W_n(t,z)$, for all $(t,z) \in [0,T] \times \R^n$, ordinary Gaussian integration by parts yields
\[
\E\left[F(Z^n)M_T\right]=\E\left[F(Z^n) \sum_{j=1}^n
\left(
Z^n_j\int_0^Te_j^n(t)W_n(t,Z^n)\dd t
-\nu\int_0^Te_j^n(t)k_j^n(t)W_n(t,Z^n)\dd t\right)\right].
\]
Finally, we deduce
\[
\begin{aligned}
\E[M_T\mid Z^n=z]
&=
\sum_{j=1}^n
\left[
z_j\int_0^Te_j^n(t)W_n(t,z)\dd t
-\nu\int_0^Te_j^n(t)k_j^n(t)W_n(t,z)\dd t
\right] \\
&=
\int_0^TW_n(t,z)
\bigl(h_z^n(t)-\nu\kappa_n(t)\bigr)\dd t.
\end{aligned}
\]
The right-hand side is continuous in $z$ and therefore defines a continuous
version of the conditional expectation, which proves
\eqref{eq:conditional-mean-critical}. For $t\in I_j^n$, the fractional form of $K_H$ gives
\begin{equation}
\label{eq:kappa-fractional-diagonal}
\kappa_n(t)
=
\frac{(t-(j-1)\Delta_n)^\alpha}
{\Gamma(\alpha+1)\Delta_n}.
\end{equation}
Thus, on the right half
$\left((j-\tfrac12)\Delta_n,j\Delta_n\right],$
one has
\begin{equation}
\label{eq:kappa-lower-half-block}
\kappa_n(t)
\ge
\frac{2^{-\alpha}}{\Gamma(\alpha+1)}
\Delta_n^{\alpha-1}.
\end{equation}
Because $\alpha<1$, this lower bound diverges as the mesh tends to zero.

\medskip

\noindent\textbf{Step 3: forcing a constant positive control.}
We set
\[
\forall t \in [0,T], \quad \overline K(t):=(K_H*1)(t)=\frac{t^\alpha}{\Gamma(\alpha+1)},
\]
and, for $n\ge2$,
\[
R_n:=\frac{\log n}{\nu \overline K(T/2)}.
\]
We define the event
\[
A_n
:=
\bigcap_{j=1}^n
\left\{
R_n\Delta_n^{1/2}
\le Z_j^n
\le (R_n+1)\Delta_n^{1/2}
\right\} \in \mathcal{F}.
\]
On $A_n$, one has for every $t\in(0,T]$,
\begin{equation}
\label{eq:constant-control-bounds}
R_n\le h_{Z^n}^n(t)\le R_n+1.
\end{equation}
Since $K_H$ is non-negative, \eqref{eq:constant-control-bounds} gives
\[
R_n\overline K(t)
\le
g_{Z^n}^n(t)
=(K_H*h_{Z^n}^n)(t)
\le
(R_n+1)\overline K(t).
\]
It follows that
\begin{equation}
\label{eq:constant-control-comparison}
\forall t \in [0,T], \quad  e^{\nu R_n\overline K(t)}
\le
e^{\nu g_{Z^n}^n(t)}
\le
e^{\nu\overline K(T)}e^{\nu R_n\overline K(t)}.
\end{equation}
Since $\overline K$ is increasing, the map
$t\mapsto e^{\nu R_n\overline K(t)}$ is increasing. Hence,
\begin{equation}
\label{eq:constant-control-integral-lower}
\int_0^Te^{\nu g_{Z^n}^n(t)}\dd t
\ge \int_{0}^{T} e^{\nu\overline K(t)R_n} \dd t \ge \int_{\frac{T}{2}}^{T} e^{\nu\overline K(t)R_n} \dd t \ge
\frac{T}{2}e^{\nu\overline K(T/2)R_n}
=
\frac{T}{2}n.
\end{equation}
Let
\[
E_n
:=
\bigcup_{j=1}^n
\left((j-\tfrac12)\Delta_n,j\Delta_n\right].
\]
With the same monotonicity argument, we get
\[
\int_{E_n}e^{\nu R_n\overline K(t)}\dd t
\ge
\frac12\int_0^Te^{\nu R_n\overline K(t)}\dd t.
\]
Using \eqref{eq:constant-control-comparison} and the previous inequality, we obtain
\[
\begin{aligned}
\int_{E_n}e^{\nu g_{Z^n}^n(t)}\dd t
\ge
\int_{E_n}e^{\nu R_n\overline K(t)}\dd t
\ge
\frac12\int_0^Te^{\nu R_n\overline K(t)}\dd t
\ge
\frac12e^{-\nu\overline K(T)}
\int_0^Te^{\nu g_{Z^n}^n(t)}\dd t.
\end{aligned}
\]
Combining this estimate with
\eqref{eq:kappa-lower-half-block} yields
\begin{equation}
\label{eq:trace-integral-lower}
\int_0^T
\kappa_n(t)e^{\nu g_{Z^n}^n(t)}\dd t
\ge
c\Delta_n^{\alpha-1}
\int_0^Te^{\nu g_{Z^n}^n(t)}\dd t,
\end{equation}
with $c=\frac{2^{-(\alpha+1)}}{\Gamma(\alpha+1)}e^{-\nu\frac{T^\alpha}{\Gamma(\alpha +1)}}>0$.

\medskip

\noindent\textbf{Step 4: the diagonal term dominates the Gaussian cost.}
Using \eqref{eq:conditional-mean-critical},
\eqref{eq:conditional-weight-bounds}, \eqref{eq:constant-control-bounds}, and
\eqref{eq:trace-integral-lower}, we deduce that, on $A_n$,
\[
\E[M_T\mid Z^n]
\le
\left(c_{+}(R_n+1)-c_{-}c \nu\Delta_n^{\alpha-1}\right)\int_0^Te^{\nu g_{Z^n}^n(t)}\dd t.
\]
Since
\[
\frac{\Delta_n^{\alpha-1}}{R_n+1}
=
T^{\alpha-1}\frac{n^{1-\alpha}}{R_n+1}
\longrightarrow\infty,
\]
\eqref{eq:constant-control-integral-lower} yields
\begin{equation}
\label{eq:conditional-mean-negative}
\E[M_T\mid Z^n]
\le
-C_1 n^{1-\alpha}\int_0^Te^{\nu g_{Z^n}^n(t)}\dd t
\le
-C_2 n^{2-\alpha},
\quad\text{on }A_n,
\end{equation}
for all sufficiently large $n$, and $C_1,C_2>0$ independent of $n$. It remains to estimate the probability of
$A_n$. Since the coordinates of $Z^n$ are independent standard normal
variables, one has
\[
\begin{aligned}
\Prob(A_n)
\ge
\left(
\Delta_n^{1/2}
\inf_{R_n\Delta_n^{1/2}\le x\le(R_n+1)\Delta_n^{1/2}}
\frac{e^{-x^2/2}}{\sqrt{2\pi}}
\right)^n \ge
\left(\frac{1}{\sqrt{2 \pi}}\Delta_n^{1/2}\right)^n
\exp\left(
-\frac{n}{2}(R_n+1)^2\Delta_n
\right).
\end{aligned}
\]
Since $n\Delta_n=T$ and $R_n=O(\log n)$, we obtain
\begin{equation}
\label{eq:gaussian-tube-probability}
\Prob(A_n)
\ge
\exp(-Cn\log n),
\end{equation}
with $C>0$ some positive constant which does not depend on $n$. Since $x\mapsto e^{-ax}$ is convex, conditional Jensen's inequality and
\eqref{eq:conditional-mean-negative} give
\[
\begin{aligned}
\E[e^{-aM_T}]
\ge
\E\left[
\exp\left(-a\E[M_T\mid Z^n]\right)
\right]\ge
\Prob(A_n)\exp\left(aC_2 n^{2-\alpha}\right).
\end{aligned}
\]
Together with \eqref{eq:gaussian-tube-probability}, this yields
\[
\log\E[e^{-aM_T}]
\ge
aC_2 n^{2-\alpha}-Cn\log n.
\]
Since $\alpha<1$, the right-hand side tends to $+\infty$ as $n\to\infty$.
This proves
\eqref{eq:negative-laplace-explosion}, and then
\eqref{eq:critical-conditioning-identity} proves
$\E[S_T^{p_\rho}]=\infty$ in \myref{item:critical-rough-kernel}. The strict
subcritical finiteness follows from Theorem \ref{thm:bergomi-main} for all three
kernel classes, while the strict supercritical explosion follows from Lemma
\ref{lem:osgood-supercritical-explosion}. Indeed, the lognormal response
satisfies the Osgood condition of that lemma. This proves the displayed
classification and completes the proof.
\end{proof}

\section{Absence of a Feller criterion in rough Heston: Proof of Theorem \ref{thm:heston-main}}

We provide two heuristics explaining this phenomenon.
\begin{enumerate}
\item Let us first explain why one should not expect a Feller-type non-attainment
condition in the fractional rough case. In the classical CIR diffusion
\[
\dd X_t=(b-\kappa X_t)\dd t+\nu\sqrt{X_t}\dd W_t,
\]
the behaviour near zero is governed by the competition between the inward drift
$b\dd t$ and the square-root noise $\nu\sqrt{X_t}\dd W_t$. Over a short time
interval of length $h$, and when the process is close to a small level $x$,
one has, at the heuristic level,
\[
\text{inward deterministic push}
\ \simeq\ 
b h,
\quad
\text{noise variance}
\ \simeq\ 
\nu^2 x h.
\]
After factoring out the small level $x$ from the squared noise coefficient, the local ratio between the inward drift and the squared noise coefficient
is independent of the time scale:
$\frac{2 b h}{\nu^2 h}
=
\frac{2b}{\nu^2}.$ The classical Feller condition $2b\ge\nu^2$ says precisely that this ratio is
large enough to prevent the square-root diffusion from reaching the boundary. For the rough Heston variance, the same comparison has to be made after the
fractional Volterra smoothing. Ignoring the lower-order mean-reversion term
$-\kappa V$ near the zero boundary, and looking over a short interval of length
$h$, the deterministic source term $b$ contributes at the scale
\[
b\int_0^h K_H(s)\dd s
=
\frac{b}{\Gamma(H+\frac32)}h^{H+\frac12}
\simeq
h^{H+\frac12}.
\]
On the noise side, suppose heuristically that, during this short interval, the
process stays close to a small level $x$. The new stochastic contribution over
the interval has the form
\[
\nu\int_0^h K_H(h-s)\sqrt{x}\,\dd W_s .
\]
By It\^o's isometry, its variance is
\[
\nu^2 x\int_0^h K_H(s)^2\dd s
=
\frac{\nu^2 x}{2H\Gamma(H+\frac12)^2}h^{2H}.
\]
Therefore,
\[
\frac{
2b\int_0^h K_H(s)\dd s
}{
\nu^2\int_0^h K_H(s)^2\dd s
}
\simeq h^{\frac12-H}
\longrightarrow0,
\quad h\downarrow0,
\]
since $H<\frac12$.
Hence, at very small time scales, the singularity of the kernel makes the
short-time variance of the stochastic convolution dominate the deterministic
inward push. This is the opposite of what happens in the classical CIR model, where the
corresponding ratio is the constant $\frac{2b}{\nu^2}$.

\item A second heuristic comes from the nonsingular Volterra--CIR boundary
criteria of \citet{BondiPulido2024}. In our notation, for the regular-kernel
equation
\[
V_t
=
v_0+\int_0^tK(t-s)(b-\kappa V_s)\dd s
+
\nu\int_0^tK(t-s)\sqrt{V_s}\dd W_s,
\]
their sufficient non-attainment condition is
$2b\ge K(0)\nu^2.$
For the fractional rough Heston kernel, one has $K_H(0+)=\infty$. Thus, along
smooth-kernel approximations of $K_H$, the condition becomes harder and harder
to satisfy as $K(0)$ grows. In the rough limit, the finite-$K(0)$ mechanism
underlying the regular Volterra--CIR Feller criterion therefore degenerates,
which is consistent with boundary attainment in the fractional rough case.
\end{enumerate}
\medskip

The preceding heuristic arguments are not proofs, since the rough Heston process
is not a semimartingale and no one-dimensional boundary test is directly
available. We make it rigorous through the affine Riccati--Volterra transform.
\begin{lemma}\label{lem:laplace-resolvent}
For every $\lambda>0$ and every $T>0$, there is a unique locally integrable
non-negative solution $y_\lambda$ to
\begin{equation}\label{eq:riccati-y}
\forall t \in (0,T], \quad
y_\lambda(t)
=
\lambda K_H(t)
-
\int_0^tK_H(t-s)
\left(
\kappa y_\lambda(s)+\frac{\nu^2}{2} y_\lambda(s)^2
\right)\dd s ,
\end{equation}
and
\begin{equation}\label{eq:laplace-resolvent}
\E[e^{-\lambda V_T}]
=
\exp\left(
-v_0(\mathcal R_H*y_\lambda)(T)
-
b\int_0^T y_\lambda(s)\dd s
\right).
\end{equation}
\end{lemma}

\begin{proof}
The affine Volterra transform formula of \citet{AbiJaberLarssonPulido2019},
applied to the scalar square-root case, gives existence and uniqueness of the
non-negative Riccati--Volterra solution $y_\lambda$ and the representation
\[
\E[e^{-\lambda V_T}]
=
\exp\left(
-v_0\lambda
+
v_0\int_0^T
\left(
\kappa y_\lambda(s)+\frac{\nu^2}{2}y_\lambda(s)^2
\right)\dd s
-
b\int_0^T y_\lambda(s)\dd s
\right).
\]
Convolving \eqref{eq:riccati-y} with $\mathcal R_H$ and using
$\mathcal R_H*K_H=1$ gives
\[
(\mathcal R_H*y_\lambda)(T)
=
\lambda
-
\int_0^T
\left(
\kappa y_\lambda(s)+\frac{\nu^2}{2}y_\lambda(s)^2
\right)\dd s.
\]
\end{proof}
In the remainder, we use the following notation for the Riemann--Liouville derivative, for $a \in [0,t)$:
\[
\D_a^{H+\frac12} f(t):=
\frac1{\Gamma(\frac12-H)}
\frac{\dd}{\dd t}
\int_a^t(t-s)^{-H-\frac12}f(s)\dd s,
\]
and we denote $\D^{H + \frac 12}:=\D_0^{H + \frac 12}$.

We need an upper bound on $y_\lambda$ which is uniform in $\lambda$, in
order to let $\lambda\to\infty$ in the Laplace transform. The key observation
is that, after applying the Riemann--Liouville derivative to
\eqref{eq:riccati-y}, the singular source term $\lambda K_H$ disappears and
the equation becomes independent of $\lambda$:
\[
\D^{H+\frac12}y_\lambda(t)
=
-\kappa y_\lambda(t)-\frac{\nu^2}{2}y_\lambda(t)^2.
\]

\begin{lemma}\label{lem:barrier}
We have
\begin{equation}\label{eq:barrier-bound}
\forall A>\frac{4H\Gamma(\frac12-H)}{\nu^2\Gamma(1-2H)}, \quad \forall \lambda>0,\quad \forall t>0,\quad
0\leq y_\lambda(t)\leq At^{-H-\frac12}.
\end{equation}
\end{lemma}
The proof of this key estimate is postponed in Section \ref{subsec:barrier}. We now provide the proof of Theorem \ref{thm:heston-main}.

\begin{proof}[Proof of Theorem \ref{thm:heston-main}]
We fix
$A>\frac{4H\Gamma(\frac12-H)}{\nu^2\Gamma(1-2H)}$. Lemma
\ref{lem:barrier} gives, for every $\lambda>0$,
\[
\int_0^T y_\lambda(s)\dd s\leq A\int_0^T s^{-H-\frac12}\dd s=\frac{A}{\frac12-H}T^{\frac12-H}, \]
and
\[
(\mathcal R_H*y_\lambda)(T)\leq \frac{A}{\Gamma(\frac12-H)}\int_0^T(T-s)^{-H-\frac12}s^{-H-\frac12}\dd s.
\]
The beta-function identity gives
\[
\int_0^T(T-s)^{-H-\frac12}s^{-H-\frac12}\dd s=T^{-2H}\frac{\Gamma(\frac12-H)^2}{\Gamma(1-2H)},
\]
and therefore
\[
(\mathcal R_H*y_\lambda)(T)\leq A\frac{\Gamma(\frac12-H)}{\Gamma(1-2H)}T^{-2H}.
\]
Lemma \ref{lem:laplace-resolvent} implies, uniformly in $\lambda>0$,
\[
\E[e^{-\lambda V_T}]
\geq
\exp\left(
-v_0A\frac{\Gamma(\frac12-H)}{\Gamma(1-2H)}T^{-2H}
-b\frac{A}{\frac12-H}T^{\frac12-H}
\right)
>0.
\]
Since $V_T\geq0$, one has $e^{-\lambda V_T}\to\1_{\{V_T=0\}}$ as $\lambda\to\infty$, and dominated convergence gives
\[
\Prob(V_T=0)
=\lim_{\lambda\to\infty}\E[e^{-\lambda V_T}]
\geq
\exp\left(
-v_0A\frac{\Gamma(\frac12-H)}{\Gamma(1-2H)}T^{-2H}
-b\frac{A}{\frac12-H}T^{\frac12-H}
\right)
>0.
\]
Since this holds for every $A>\frac{4H\Gamma(\frac12-H)}{\nu^2\Gamma(1-2H)}$, letting $A\downarrow \frac{4H\Gamma(\frac12-H)}{\nu^2\Gamma(1-2H)}$
gives \eqref{eq:lower-bound-main}. 
\end{proof}

\subsection{Proof of Lemma \ref{lem:barrier}}
\label{subsec:barrier}
This is a purely deterministic analytic result based on the fractional maximum principle. The proof of Lemma \ref{lem:barrier} relies on the following two auxiliary results. Lemma \ref{lem:prelim-y} is proved afterwards, while Lemma \ref{lem:extremum} is a standard endpoint maximum principle for Riemann--Liouville derivatives.

\begin{lemma}\label{lem:prelim-y}
For every $\lambda>0$, the function $y_\lambda$ satisfies the following
properties.
\begin{enumerate}[label=(\roman*),ref=(\roman*)]

\myitem{item:y-lambda-origin} One has
\begin{equation}\label{eq:small-origin}
\frac{y_\lambda(t)}{t^{-H-\frac12}}\longrightarrow0, \quad t\downarrow0,
\end{equation}

\myitem{item:y-lambda-fractional-ode} One has
$y_\lambda\in C^1_{\mathrm{loc}}((0,\infty))$ and
\begin{equation}\label{eq:fractional-ode-y}
\forall t>0,\quad
\D^{H+\frac12} y_\lambda(t)
=
-\kappa y_\lambda(t)-\frac{\nu^2}{2} y_\lambda(t)^2.
\end{equation}
\end{enumerate}
\end{lemma}

\begin{lemma}
\label{lem:extremum}
Let $0<H<1/2$, $t_0<t_1$, and $f\in C^1([t_0,t_1])$. If
\[
f(t)\le 0,\quad t\in[t_0,t_1],
\quad\text{and}\quad
f(t_1)=0,
\]
then
\[
\D_{t_0}^{H+\frac12} f(t_1)\ge0.
\]
\end{lemma}
This elementary endpoint lemma is a special case of the extremum principle for
Riemann--Liouville derivatives; see
\citet{AlRefai2012,AlRefaiLuchko2014}.

We now provide the proof of Lemma \ref{lem:barrier}.
\begin{proof}[Proof of Lemma \ref{lem:barrier}]
We define $\bar y(t):=At^{-H-\frac12}$ for $t \in (0,T]$. A direct computation gives
\[
\D^{H+\frac12}\bar y(t)
=
-A\frac{2H\Gamma(\frac12-H)}{\Gamma(1-2H)}t^{-2H-1}.
\]
Since $A>\frac{4H\Gamma(\frac12-H)}{\nu^2\Gamma(1-2H)}$, we have, for all $t>0$,
\begin{equation}\label{eq:strict-super}
D^{H+\frac12}\bar y(t)
-\left(-\kappa\bar y(t)-\frac{\nu^2}{2}\bar y(t)^2\right)=
A\left(
\frac{\nu^2}{2}A
-\frac{2H\Gamma(\frac12-H)}{\Gamma(1-2H)}
\right)t^{-2H-1}
+\kappa At^{-H-\frac12}
>0.
\end{equation}
The limit \eqref{eq:small-origin} implies $y_\lambda(t)-\bar y(t)<0$ for all sufficiently small $t>0$. If there were a time $t>0$ such that $y_\lambda(t)>\bar y(t)$, continuity would yield a first contact time $t_*>0$ such that
\[
w(t):=y_\lambda(t)-\bar y(t)<0\quad\text{for }0<t<t_*,
\quad
w(t_*)=0.
\]
For every $\varepsilon\in(0,t_*)$, Lemma \ref{lem:prelim-y} gives $w\in C^1([\varepsilon,t_*])$, and Lemma \ref{lem:extremum} applied to $w$ on $[\varepsilon,t_*]$ gives
$\D_\varepsilon^{H+\frac12} w(t_*)\geq0.$
We have
\begin{align*}
\D_0^{H+\frac12} w(t_*)
&=
\frac{1}{\Gamma(\frac12-H)}
\frac{\dd}{\dd t}\bigg|_{t=t_*}
\int_0^t (t-s)^{-H-\frac12}w(s)\dd s \\
&=
\D_\varepsilon^{H+\frac12} w(t_*)
+
\frac{1}{\Gamma(\frac12-H)}
\frac{\dd}{\dd t}\bigg|_{t=t_*}
\int_0^\varepsilon (t-s)^{-H-\frac12}w(s)\dd s \\
&=
\D_\varepsilon^{H+\frac12} w(t_*)
-\frac{H+\frac12}{\Gamma(\frac12-H)}
\int_0^\varepsilon(t_*-s)^{-H-\frac32}w(s)\dd s .
\end{align*}
Since $w(s)<0$ on $(0,\varepsilon]$, we obtain $\D_0^{H+\frac12} w(t_*)\geq0$. On the other hand, \eqref{eq:fractional-ode-y}, the equality $y_\lambda(t_*)=\bar y(t_*)$, and \eqref{eq:strict-super} give
\[
\D_0^{H+\frac12} w(t_*)=\left(-\kappa\bar y(t_*)-\frac{\nu^2}{2}\bar y(t_*)^2\right)-\D^{H+\frac12}\bar y(t_*)<0.
\]
This contradiction proves the desired result.
\end{proof}

We now provide the proof of Lemma \ref{lem:prelim-y}.

\begin{proof}[Proof of Lemma \ref{lem:prelim-y}]
The non-negativity of $y_\lambda$ follows from Lemma
\ref{lem:laplace-resolvent}. Since the convolution term in
\eqref{eq:riccati-y} is non-positive, it follows that
\[
\forall t>0, \quad 0\le y_\lambda(t)\le \lambda K_H(t).
\]
Hence,
\[
y_\lambda(t)t^{H+\frac12}
\longrightarrow0,
\quad t\downarrow0,
\]
and hence \eqref{eq:small-origin}. Writing
\[
y_\lambda(t)
=
\lambda K_H(t)
+
\frac1{\Gamma(H+\frac12)}
\int_0^t(t-s)^{H-\frac12}
\left(-\kappa y_\lambda(s)-\frac{\nu^2}{2}y_\lambda(s)^2\right)\dd s,
\]
we obtain the local regularity from the standard Abel regularization theorem:
the Riemann--Liouville integral of order $H+\frac12$ maps
$L^\infty_{\mathrm{loc}}$ into $C^{H+\frac12}_{\mathrm{loc}}$, and maps
$C^{H+\frac12}_{\mathrm{loc}}$ into $C^{2H+1}_{\mathrm{loc}}$; see
\citet[Theorem~4.2.1, p.~70]{GorenfloVessella1991} or
\citet[Theorem~3.1]{CarloneFiorenzaTentarelli2017}. Since $K_H$ is smooth
away from the origin and $y_\lambda$ is locally bounded on $(0,\infty)$, it
thus follows that
$y_\lambda\in C^{2H+1}_{\mathrm{loc}}((0,\infty)).$ As $2H+1>1$, this gives
$y_\lambda\in C^1_{\mathrm{loc}}((0,\infty)).$
Applying the Riemann--Liouville fractional integral of order $\frac12-H$ to
\eqref{eq:riccati-y}, using that its application to $K_H$ is equal to $1$,
gives
\[
\frac1{\Gamma(\frac12-H)}
\int_0^t(t-s)^{-H-\frac12}y_\lambda(s)\dd s
=
\lambda
+
\int_0^t
\left(
-\kappa y_\lambda(s)-\frac{\nu^2}{2}y_\lambda(s)^2
\right)\dd s.
\]
Differentiating yields
\[
\forall t>0, \quad \D^{H+\frac12}y_\lambda(t)
=
-\kappa y_\lambda(t)-\frac{\nu^2}{2}y_\lambda(t)^2.
\]
\end{proof}

\section{Conclusion}

We have solved two open problems from the rough-volatility literature. This
suggests several directions for future research. A first question is whether the
Feller-type invariance and boundary-exit criteria of
\citet{BondiPulido2024}, established for smooth-kernel Volterra equations, can
be extended to rough kernels. Such an extension could in particular provide a
structural explanation for the boundary attainment phenomenon in rough Heston. It would also be interesting to understand the stability of
Feller-type criteria when passing from smooth kernels to rough kernels. As
suggested by the Volterra--CIR heuristics discussed above, the limiting
rough-kernel regime may turn a non-attainment criterion for smooth kernels into
boundary attainment. 
Finally, a related and substantially harder open problem is pathwise uniqueness
for the rough Heston equation. As in the classical Heston/CIR case, the key issue should be the occupation
behaviour at the zero boundary, namely whether the process leaves zero
immediately after hitting it or spends a non-negligible amount of time there.

\appendix
\renewcommand{\thesection}{\Alph{section}}
\renewcommand{\thesubsection}{\arabic{subsection}}
\renewcommand{\theequation}{\thesection.\arabic{equation}}

\section{Gassiat's Osgood criterion}
\label{subsec:app-osgood-supercritical}

We record the strictly supercritical moment criterion in the form used in the
main text, directly under the standing kernel and volatility assumptions. Set
\begin{equation}\label{eq:kernel-primitive}
\forall t\in[0,T], \quad \overline K(t):=\int_0^t K(s)\dd s,
\end{equation}
and define its generalized inverse by
\begin{equation}\label{eq:kernel-primitive-generalized-inverse}
\Phi_K(u)
:=
\inf\bigl\{t\in[0,T]:\overline K(t)\ge u\bigr\},
\quad u\ge0,
\end{equation}
with the convention $\inf\varnothing=\infty$. Assumption
\ref{ass:kernel} implies that $\overline K(t)>0$ for every $t>0$. Indeed,
associativity and \eqref{eq:resolvent-identity} give
\begin{equation}\label{eq:kernel-primitive-resolvent-bound}
t
=
(\mathcal R_K*\overline K)(t)
\le
\mathcal R_K([0,T])\overline K(t),
\quad t\in[0,T].
\end{equation}

\begin{lemma}[Gassiat's Osgood criterion for strictly supercritical moments]
\label{lem:osgood-supercritical-explosion}
Let $\rho\in(-1,0)$ and $p>p_\rho$, and suppose that Assumptions
\ref{ass:kernel} and \ref{ass:volatility} are in force. Assume that there exist
\[
T_0\in(0,T],
\quad
\vartheta
\in
\left(
0,\nu\bigl(\sqrt{p(p-1)}+p\rho\bigr)
\right),
\quad
A>0,
\]
and a continuous, non-decreasing, locally Lipschitz function
$b:[A,\infty)\to(0,\infty)$ such that
\begin{equation}\label{eq:controlled-volatility-tail-lower-bound}
\vartheta\sigma(t,x)\ge b(x),
\quad
(t,x)\in[0,T_0]\times[A,\infty).
\end{equation}
If
\begin{equation}\label{eq:kernel-osgood-condition}
\int_A^\infty
\Phi_K\left(\frac{x}{b(x)}\right)
\frac{\dd x}{x}
<\infty,
\end{equation}
then
\[
\E[S_{T_0}^p]=\infty.
\]
\end{lemma}
\citet[Theorem~2]{Gassiat2019} proves the result for the
Riemann--Liouville kernel. Lemma
\ref{lem:osgood-supercritical-explosion} is the immediate adaptation to
kernels satisfying Assumption \ref{ass:kernel}: the support argument
follows from \eqref{eq:support-C0}, and the deterministic level-crossing
argument applies with $\overline K$ and $\Phi_K$. We therefore omit the
proof.

\section{Technical estimates for the Volterra--Bergomi model}
\label{subsec:app-volterra-gaussian-cm}

In this appendix, we prove the technical lemmas used in Section \ref{sec:thm1}
to establish Theorem \ref{thm:bergomi-main}.
\begin{proof}[Proof of Lemma \ref{lem:volterra-gaussian-cm}]
We fix $f\in L^2([0,T])$. For $0\le t\le t+h\le T$,
\[
\begin{aligned}
|(K*f)(t+h)-(K*f)(t)|
&\le
\left(\int_0^h K(r)^2\dd r\right)^{1/2}\|f\|_{L^2}
+
\left(\int_0^t |K(r+h)-K(r)|^2\dd r\right)^{1/2}\|f\|_{L^2}.
\end{aligned}
\]
Assumption \myref{item:kernel-modulus} therefore gives a continuous
representative of $K*f$, vanishing at zero. This proves
\myref{item:volterra-deterministic-continuity}.

For $G_t=\nu\int_0^tK(t-s)\dd B_s$, It\^o's isometry and the same estimate give
\[
\E[|G_{t+h}-G_t|^2]\le C h^{2\gamma_K}.
\]
Since $G_{t+h}-G_t$ is Gaussian, all higher moments satisfy the corresponding
power bounds. Kolmogorov's criterion yields a continuous version, with
$G_0=0$. This proves \myref{item:volterra-gaussian-continuity}.

The preceding continuity shows that the map
\[
f\longmapsto\nu K*f
\]
from $L^2([0,T])$ to $C_0([0,T])$ is continuous and linear. The
Cameron--Martin space of the Gaussian image is
$\{\nu K*f:f\in L^2([0,T])\}.$
Moreover, if $K*f=0$, then convolving with $\mathcal R_K$ gives
\[
1*f=\mathcal R_K*(K*f)=0,
\]
hence $f=0$ a.e. Thus the Cameron--Martin norm is
$\|\nu K*f\|_{\cH_G}=\|f\|_{L^2([0,T])}$.

It remains to identify the support. The support of a Gaussian measure is the
closure of its Cameron--Martin space in the ambient Banach space. For every
$\varphi\in C^1([0,T])$ with $\varphi(0)=0$, the function
$\nu^{-1}\mathcal R_K*\varphi'$ belongs to $L^2([0,T])$ because
$\mathcal R_K$ is finite and $\varphi'$ is continuous. Moreover,
\[
\nu K*\bigl(\nu^{-1}\mathcal R_K*\varphi'\bigr)
=K*(\mathcal R_K*\varphi')
=(K*\mathcal R_K)*\varphi'
=1*\varphi'
=\varphi.
\]
Hence $C^1_0([0,T])$ is contained in the Cameron--Martin space. Since
$C^1_0([0,T])$ is dense in $C_0([0,T])$, we obtain
$\supp\mu=C_0([0,T])$. This proves
\myref{item:volterra-cm-space}.
\end{proof}

\begin{proof}[Proof of Lemma \ref{lem:borell}]
We set
\[
A_\alpha:=\{g\in C_0([0,T]) ~|~\|g\|_\infty\le\alpha\}.
\]
By \eqref{eq:support-C0}, we have $0 \in C_0([0,T])=\operatorname{supp} \mu$. Hence, we deduce $\mu(A_\alpha)>0$. If $\mu(A_\alpha)=1$, the
claim is immediate from the assumption with $f=0$ and $R \in (C,\infty)$. We may therefore assume $\mu(A_\alpha)\in(0,1)$. By \myref{item:volterra-cm-space} of Lemma
\ref{lem:volterra-gaussian-cm},
\[
\left\{\nu K*f:\ \|f\|_{L^2([0,T])}\le r\right\}
=
r\mathbb B_{\cH_G}:=\{ g  \in  \cH_G ~|~ \|g\|_{\cH_G} \le r\}.
\]
Borell's Gaussian isoperimetric inequality, see
\citet[Theorem 4.3]{Ledoux1996}, gives
\begin{align}
\label{eq:isoperimetric}
\mu_*(A_\alpha+r\mathbb B_{\cH_G})
\ge
\Phi\left(\Phi^{-1}(\mu(A_\alpha))+r\right),
\end{align}
where $\mu_*$ denotes the inner measure associated with $\mu$, and
$\Phi$ denotes the standard normal distribution function. With $a_\alpha:=\Phi^{-1}(\mu(A_\alpha))$,
for every $\delta\in(0,1/2)$, there exists a constant
$C_\delta<\infty$ such that, for all $r\ge0$,
\begin{align}
\label{eq:gaussian_bound}
1-\Phi(a_\alpha+r)
\le
C_\delta
\exp\left(-\left(\frac12-\delta\right)r^2\right).
\end{align}
Indeed, when $a_\alpha+r\ge0$, the standard bound
$1-\Phi(x)\le \exp(-x^2/2)$, $x\ge0$, yields
\[
1-\Phi(a_\alpha+r)
\le
\exp\left(-\frac{(a_\alpha+r)^2}{2}\right).
\]
Moreover, by Young's inequality,
\[
\frac{(a_\alpha+r)^2}{2}
\ge
\left(\frac12-\delta\right)r^2-\frac{a_\alpha^2}{4\delta}.
\]
Thus the desired estimate follows in this case. If $a_\alpha+r<0$, then
$r$ ranges over a bounded interval, and the same estimate follows after
increasing the constant $C_\delta$.
For $R>C$, the assumed Cameron--Martin estimate gives
\[
A_\alpha+\left(\frac{R-C}{B}\right)^{1/2}\mathbb B_{\cH_G}
\subset
\{\Psi\le R\}.
\]
Since $\{\Psi\le R\}$ is measurable, Borell's inequality \eqref{eq:isoperimetric} and the preceding
inclusion yield
\[
\mu(\Psi\le R)
\ge
\mu_*\left(
A_\alpha+\left(\frac{R-C}{B}\right)^{1/2}\mathbb B_{\cH_G}
\right)
\ge
\Phi\left(a_\alpha+\left(\frac{R-C}{B}\right)^{1/2}\right).
\]
Therefore, \eqref{eq:gaussian_bound} yields
\[
\mu(\Psi>R)
\le
1-\Phi\left(a_\alpha+\left(\frac{R-C}{B}\right)^{1/2}\right)
\le
C_\delta
\exp\left(-\left(\frac12-\delta\right)\frac{R-C}{B}\right).
\]
Up to changing the constant, this gives the stated estimate for all $R\ge0$.
\end{proof}

\subsection{Proof of Lemma \ref{lem:well-posed-general}}
\label{subsec:proof-cm-estimate}

The key point is that the
Volterra chain rule used below is a deterministic one, well known in the
Volterra literature. It has no direct counterpart obtained by replacing $\dd s$ with a Brownian increment $\dd B_s$, which
is the stochastic input of interest here. We therefore first work along
Cameron--Martin directions: the Brownian Volterra input is replaced by a
deterministic Cameron--Martin shift, the deterministic chain rule is applied in
that finite-energy setting, and the passage back to the Gaussian input is made
later through Borell's inequality. The proof of Lemma \ref{lem:well-posed-general} relies on the following two
lemmas, which are proved afterwards.

\begin{lemma}\label{lem:chain-rule}
Assumption \ref{ass:kernel} is in force. Let $h\in L^2([0,T])$, set
$X=K*h$, and let $F\in C^1(\R)$ be convex. Define
\[
A:=\mathcal R_K*(F(X)-F(0)).
\]
Then, for all $0\le s\le t\le T$,
\begin{equation}\label{eq:chain-rule-increments}
A(t)-A(s)
\le
\int_s^t F'(X_u)h_u\dd u.
\end{equation}
Equivalently, in the sense of Stieltjes measures,
\begin{equation}\label{eq:chain-rule-stieltjes}
\dd A(t)\le F'(X_t)h_t\dd t.
\end{equation}
\end{lemma}

\begin{remark}\label{rem:chain-rule-classical-case}
If $K\equiv 1$, then $\mathcal{R}_K=\delta_0$ and we retrieve the usual chain rule
\[
\int_0^\tau F'(X_t)h_t\dd t
=
\int_0^\tau F'(X_t)\dot X_t\dd t
=
F(X_\tau)-F(0).
\]
\end{remark}

This is the standard Volterra convexity inequality associated with positive
first-kind resolvents; see, for instance,
\citet[Chapter 18]{GripenbergLondenStaffans1990}.

\begin{lemma}
\label{lem:volterra-energy-small-shift}
Assumptions \ref{ass:kernel} and \ref{ass:volatility} are in force. For every
$\varepsilon>0$, there exist $\alpha_\varepsilon>0$ and
$C_\varepsilon<\infty$ such that, for every $g\in C([0,T])$ with
$\|g\|_\infty\le\alpha_\varepsilon$ and every $u\in L^2([0,T])$, if
\[
\forall t \in [0,T], \quad y(t)=y_0+g_t+(K*u)(t),
\quad
q(t)=\sigma(t,y(t)),
\]
then
\begin{equation}\label{eq:volterra-energy-small-shift}
\int_0^T q(t)u(t)\dd t
\ge
-C_\varepsilon
-\varepsilon\int_0^T q(t)^2\dd t
-\varepsilon\int_0^T u(t)^2\dd t.
\end{equation}
\end{lemma}

We now provide the proof of Lemma \ref{lem:well-posed-general}.

\begin{proof}[Proof of Lemma \ref{lem:well-posed-general}]
The value of the generic constant $C$ may change from line to line.

\medskip

\noindent\textbf{Proof of \myref{item:feedback-well-posedness} and
\myref{item:feedback-continuity}.}
Local existence, uniqueness, and continuous dependence up to the explosion time
follow from the standard local theory for nonlinear Volterra equations with
$L^1$-kernels and locally Lipschitz nonlinearities; see
\citet[Chapter 12]{GripenbergLondenStaffans1990} and
\citet[Appendix B]{AbiJaberLarssonPulido2019}.

Fix $x\in C([0,T])$ and set $\bar y_0:=y_0+x_0$. Choose $\theta>0$ such
that $c:=\nu a/2-\theta-2\theta\nu^2a^2>0$, and let $\alpha$ be given by
Lemma \ref{lem:volterra-energy-small-shift}, with $y_0$ replaced by
$\bar y_0$. By \eqref{eq:support-C0}, there exists $f\in L^2([0,T])$ such
that
\[
g:=x-x_0-\nu K*f,
\quad
\|g\|_\infty\le\alpha/2.
\]
Let $q(t):=\sigma(t,Y_t^x)$ and let $\tau$ be smaller than the explosion time.
Then, on $[0,\tau]$,
\[
Y^x=\bar y_0+g+K*\bigl(\nu(f-aq)\bigr).
\]
Applying the lemma on $[0,\tau]$ and using Young's inequality give
\[
c\int_0^\tau q(t)^2\dd t
\le
\left(\frac{\nu}{2a}+2\theta\nu^2\right)\|f\|_{L^2([0,T])}^2+C_\theta.
\]
Consequently,
\[
\sup_{t\le\tau}|Y_t^x|
\le
|\bar y_0|+\|g\|_\infty
+\nu\|K\|_{L^2([0,T])}
\left(\|f\|_{L^2([0,T])}+a\|q\|_{L^2(0,\tau)}\right),
\]
uniformly in $\tau$. The continuation criterion therefore excludes explosion,
which proves \myref{item:feedback-well-posedness}.

Now let $x^n\to x$ in $C([0,T])$. The same $f$ satisfies
$\|x^n-x_0^n-\nu K*f\|_\infty\le\alpha$ for all sufficiently large $n$.
Since $x_0^n$ remains in a compact set, the constants above may be chosen
uniformly in $n$. Hence the solutions $Y^{x^n}$ and $Y^x$ remain in a common
compact interval. On this interval, $\sigma$ is
Lipschitz in its second variable, uniformly in time, and therefore
\[
|Y_t^{x^n}-Y_t^x|
\le
\|x^n-x\|_\infty
+\nu aL\int_0^tK(t-s)|Y_s^{x^n}-Y_s^x|\dd s.
\]
The Volterra Grönwall inequality yields
$\|Y^{x^n}-Y^x\|_\infty\to0$, proving
\myref{item:feedback-continuity}.

\medskip

\noindent\textbf{Proof of \myref{item:feedback-coercivity}.}
We fix $B>a^{-2}$. We choose $\theta>0$ small enough so that
\[
\frac{\nu a}{2}-\theta-2\theta\nu^2a^2
>0, \quad
\frac{\frac{\nu}{2a}+2\theta\nu^2}
{\frac{\nu a}{2}-\theta-2\theta\nu^2a^2}
\le
B.
\]
This is possible because the ratio on the left tends to $a^{-2}$ as
$\theta\downarrow0$. We now assume that
$g\in C_0([0,T])$, $\|g\|_\infty\le \alpha_\theta$, and
$f\in L^2([0,T])$, where $\alpha_\theta$ is given by Lemma
\ref{lem:volterra-energy-small-shift}. We set $Y=Y^{g+\nu K*f}$ and
\[
q(t):=\sigma(t,Y_t).
\]
Then
\[
Y=y_0+g+K*\bigl(\nu(f-aq)\bigr).
\]
Applying Lemma \ref{lem:volterra-energy-small-shift} with $g$ and
$u(t)=\nu(f(t)-aq(t))$, we obtain
\[
\nu\int_0^T q(t)f(t)\dd t
-
\nu a\int_0^T q(t)^2\dd t
\ge
-C_\theta
-\theta\int_0^T q(t)^2\dd t
-\theta\nu^2\int_0^T(f(t)-aq(t))^2\dd t .
\]
Therefore, using Young's inequality,
\begin{align*}
\nu a\int_0^T q(t)^2\dd t
&\le
\nu\int_0^T q(t)f(t)\dd t
+
C_\theta
+
\theta\int_0^T q(t)^2\dd t
+
\theta\nu^2\int_0^T(f(t)-aq(t))^2\dd t
\\
&\leq
\left(
\frac{\nu a}{2}
+
\theta
+
2\theta\nu^2a^2
\right)
\int_0^T q(t)^2\dd t
+
\left(
\frac{\nu}{2a}
+
2\theta\nu^2
\right)
\int_0^T f(t)^2\dd t
+
C_\theta.
\end{align*}
Hence,
\[
\left(\frac{\nu a}{2}-\theta-2\theta\nu^2a^2\right)
\int_0^T q(t)^2\dd t
\le
\left(
\frac{\nu}{2a}
+
2\theta\nu^2
\right)
\int_0^T f(t)^2\dd t
+
C_\theta.
\]
By the choice of $\theta$,
\[
\Psi(g+\nu K*f)
=\int_0^T q(t)^2\dd t
\le
B\|f\|_{L^2([0,T])}^2
+
\frac{C_\theta}
{\frac{\nu a}{2}-\theta-2\theta\nu^2a^2}.
\]
\end{proof}

We now provide the proof of Lemma \ref{lem:chain-rule}.

\begin{proof}[Proof of Lemma \ref{lem:chain-rule}]
We choose $q_K$ right-continuous and non-increasing and define its Stieltjes
measure $\mu_K$ by
\[
\mu_K((u,v])=q_K(u)-q_K(v),
\quad 0<u<v\le T.
\]
For every $z\in C^1([0,T])$ with $z_0=0$, Stieltjes integration by parts
gives
\begin{equation}\label{eq:bv-derivative}
\dd t\text{-a.e.},\quad
\frac{\dd}{\dd t}(q_K*z)(t)
=
q_K(t)z_t
+
\int_{(0,t]}(z_t-z_{t-r})\mu_K(\dd r).
\end{equation}
We fix $U\in C^1([0,T])$ with $U_0=0$. Applying
\eqref{eq:bv-derivative} to $U$ and to $F(U)-F(0)$, and using
$\mathcal R_K*U=q_0U+q_K*U$, gives, for almost every $t$,
\[
\begin{aligned}
&F'(U_t)\frac{\dd}{\dd t}(\mathcal R_K*U)(t)
-
\frac{\dd}{\dd t}\bigl(\mathcal R_K*(F(U)-F(0))\bigr)(t) \\
&=q_K(t)\bigl[F'(U_t)U_t-F(U_t)+F(0)\bigr] +
\int_{(0,t]}
\bigl[
F'(U_t)(U_t-U_{t-r})-
F(U_t)+F(U_{t-r})
\bigr]\mu_K(\dd r)
\ge0,
\end{aligned}
\]
where the last inequality follows from convexity of $F$. Hence, in the
sense of Stieltjes measures,
\begin{equation}\label{eq:smooth-chain-measure}
\dd\bigl(\mathcal R_K*(F(U)-F(0))\bigr)(t)
\le
F'(U_t)\frac{\dd}{\dd t}(\mathcal R_K*U)(t)\dd t.
\end{equation}
Equivalently, for all $0\le s\le t\le T$,
\begin{equation}\label{eq:smooth-chain-increments}
\bigl(\mathcal R_K*(F(U)-F(0))\bigr)(t)
-
\bigl(\mathcal R_K*(F(U)-F(0))\bigr)(s)
\le
\int_s^t
F'(U_u)\frac{\dd}{\dd u}(\mathcal R_K*U)(u)\dd u.
\end{equation}
It remains to pass from smooth paths to $X=K*h$. We extend all functions by
zero to $(-\infty,0)$ and choose a non-negative
$\zeta\in C_c^\infty(0,1)$ with $\int_0^1\zeta(s)\dd s=1$. We set
$\zeta_\varepsilon(t)=\varepsilon^{-1}\zeta(t/\varepsilon)$ and define
$X^\varepsilon=\zeta_\varepsilon*X$. Then
$X^\varepsilon\in C^1([0,T])$, $X^\varepsilon_0=0$, and
$X^\varepsilon\to X$ uniformly on $[0,T]$. Moreover, since
$X=K*h$ and $\mathcal R_K*K=1$, associativity gives
\[
\mathcal R_K*X^\varepsilon
=
\zeta_\varepsilon*(\mathcal R_K*X)
=
\zeta_\varepsilon*(1*h),
\quad
\frac{\dd}{\dd t}(\mathcal R_K*X^\varepsilon)(t)
=(\zeta_\varepsilon*h)(t).
\]
Applying \eqref{eq:smooth-chain-increments} to $U=X^\varepsilon$, we obtain,
for all $0\le s\le t\le T$,
\[
\begin{aligned}
&\bigl(\mathcal R_K*(F(X^\varepsilon)-F(0))\bigr)(t)
-
\bigl(\mathcal R_K*(F(X^\varepsilon)-F(0))\bigr)(s) \le
\int_s^t
F'(X^\varepsilon_u)(\zeta_\varepsilon*h)(u)\dd u.
\end{aligned}
\]
Letting $\varepsilon\downarrow0$, using
$\zeta_\varepsilon*h\to h$ in $L^2([0,T])$, the uniform convergence of
$X^\varepsilon$, the continuity of $F'$ on bounded sets, and the finite
mass of $\mathcal R_K$, gives
\[
A(t)-A(s)
\le
\int_s^t F'(X_u)h_u\dd u,
\quad 0\le s\le t\le T,
\]
where $A$ is the function defined in the statement. This proves \eqref{eq:chain-rule-increments}. Equivalently, the function
$t\longmapsto \int_0^t F'(X_u)h_u\dd u-A(t)$
is non-decreasing, which is exactly
\[
\dd A(t)\le F'(X_t)h_t\dd t
\]
in the sense of Stieltjes measures. This proves
\eqref{eq:chain-rule-stieltjes}.
\end{proof}

We now provide the proof of Lemma \ref{lem:volterra-energy-small-shift}.

\begin{proof}[Proof of Lemma \ref{lem:volterra-energy-small-shift}]
Throughout the proof, constants may change from line to line. We first record
two elementary consequences of Assumption \ref{ass:volatility}. Taking
$\delta=m_-/2$ in \eqref{eq:r-negligible}, we get
\begin{equation}\label{eq:s-q-comparison-small-shift}
|s(x)|^2+|m(t)s(x)|^2
\le C\bigl(1+|\sigma(t,x)|^2\bigr),
\quad (t,x)\in[0,T]\times\R.
\end{equation}
Moreover, \eqref{eq:s-derivative-growth} and Grönwall's lemma imply that, for
$|h|\le1$,
\begin{equation}\label{eq:s-shift-comparison-small-shift}
|s(x+h)|\le C(1+|s(x)|),
\quad
|s(x+h)-s(x)|
\le C|h|\bigl(1+|s(x+h)|+|s(x)|\bigr).
\end{equation}
Combining this with \eqref{eq:s-q-comparison-small-shift} gives, again for
$|h|\le1$,
\begin{equation}\label{eq:q-shift-comparison-small-shift}
|\sigma(t,x)|\le C\bigl(1+|\sigma(t,x+h)|\bigr).
\end{equation}

\medskip

\noindent\textbf{Step 1: the estimate without shift.}
We first isolate the case $g=0$, where the Volterra chain-rule applies
directly to $K*u$ and yields the desired energy estimate. The general case is
then obtained by stability, since a small continuous shift $g$ changes
$\sigma(t,y^0)$ only by an error that will be absorbed by Young's inequality.
We set
\[
y^0(t)=y_0+(K*u)(t).
\]
Applying Lemma \ref{lem:chain-rule} to $K*u$ and to
the convex function $x\mapsto\mathcal S(y_0+x)$, we have
\[
\dd A(t)\le s(y^0(t))u(t)\dd t,
\quad
A:=\mathcal R_K*(\mathcal S(y^0)-\mathcal S(y_0)).
\]
Multiplying by the positive weight $m$ gives
\[
\int_0^T m(t)s(y^0(t))u(t)\dd t
\ge
\int_0^T m(t)\dd A(t).
\]
Since $A(0)=0$, Stieltjes integration by parts gives
\[
\int_0^T m(t)\dd A(t)
=
m(T)A(T)-\int_0^T A(t)m'(t)\dd t.
\]
Since $\mathcal R_K$ is positive and $\mathcal S\ge0$, for every
$t\in[0,T]$,
\[
-\mathcal S(y_0)\mathcal R_K([0,T])
\le
A(t)
\le
\int_0^t \mathcal S(y^0(t-r))\mathcal R_K(\dd r).
\]
In particular,
\[
m(T)A(T)\ge -m_+\mathcal S(y_0)\mathcal R_K([0,T]).
\]
Writing $(m')_\pm(t):=\max\{\pm m'(t),0\}$, we obtain
\[
\begin{aligned}
-\int_0^T A(t)m'(t)\dd t
&=
-\int_0^T A(t)(m')_+(t)\dd t
+
\int_0^T A(t)(m')_-(t)\dd t \\
&\ge
-\int_0^T (m')_+(t)
\int_0^t \mathcal S(y^0(t-r))\mathcal R_K(\dd r)\dd t
-\mathcal S(y_0)\mathcal R_K([0,T])
\int_0^T(m')_-(t)\dd t.
\end{aligned}
\]
Hence, by Fubini's theorem and the finiteness of $\mathcal R_K$,
\[
\int_0^T m(t)\dd A(t)
\ge
-C
-
\left\|(m')_+\right\|_{L^\infty(0,T)}\mathcal R_K([0,T])
\int_0^T\mathcal S(y^0(t))\dd t.
\]
If $m$ is non-increasing, then $(m')_+=0$ and the last integral disappears.
Otherwise, the second alternative in
\eqref{eq:weight-lyapunov-compatibility} applies, and, for every $\delta>0$,
\[
\begin{aligned}
\int_0^T\mathcal S(y^0(t))\dd t
\le
\delta\int_0^T s(y^0(t))^2\dd t+TC_\delta
\le
\delta m_-^{-2}\int_0^T \bigl[m(t)s(y^0(t))\bigr]^2\dd t+TC_\delta,
\end{aligned}
\]
since $m(t)\ge m_-$. Thus, in both cases,
choosing $\delta$ sufficiently small when needed and absorbing the remaining
terms into a constant $C_\eta$, we obtain, for every $\eta>0$,
\[
\int_0^T m(t)\dd A(t)
\ge
-C_\eta
-\eta\int_0^T \bigl[m(t)s(y^0(t))\bigr]^2\dd t.
\]
The additional term $-\eta\int_0^T u(t)^2\dd t$ can only decrease the
right-hand side. Consequently, for every $\eta>0$,
\begin{equation}\label{eq:principal-estimate-small-shift}
\int_0^T m(t)s(y^0(t))u(t)\dd t
\ge
-C_\eta
-\eta\int_0^T \bigl[m(t)s(y^0(t))\bigr]^2\dd t
-\eta\int_0^T u(t)^2\dd t.
\end{equation}
Since
\[
q^0(t)=\sigma(t,y^0(t))=m(t)s(y^0(t))+r_\sigma(t,y^0(t)),
\]
we write
\[
\int_0^T q^0(t)u(t)\dd t
=
\int_0^T m(t)s(y^0(t))u(t)\dd t
+
\int_0^T r_\sigma(t,y^0(t))u(t)\dd t.
\]
Using \eqref{eq:r-negligible}, \eqref{eq:s-q-comparison-small-shift}, and
Young's inequality, choosing the parameter in \eqref{eq:r-negligible} small
enough gives
\[
\left|
\int_0^T r_\sigma(t,y^0(t))u(t)\dd t
\right|
\le
\eta\int_0^T q^0(t)^2\dd t
+
\eta\int_0^T u(t)^2\dd t
+
C_\eta.
\]
Combining this estimate with \eqref{eq:principal-estimate-small-shift}, and
using again \eqref{eq:s-q-comparison-small-shift} to control
$\bigl[m(t)s(y^0(t))\bigr]^2$ in terms of $1+q^0(t)^2$, we obtain, after renaming $\eta$,
\begin{equation}\label{eq:no-shift-estimate-small-shift}
\int_0^T q^0(t)u(t)\dd t
\ge
-C_\eta
-\eta\int_0^T q^0(t)^2\dd t
-\eta\int_0^T u(t)^2\dd t.
\end{equation}

\medskip

\noindent\textbf{Step 2: stability under a small continuous shift.}
We now allow a small perturbation $g$. We set
\[
y(t):=y^0(t)+g_t,
\quad
q(t):=\sigma(t,y(t)).
\]
We assume $\|g\|_\infty\le\alpha_\varepsilon\le1$, where
$\alpha_\varepsilon$ will be chosen at the end. Applying \eqref{eq:no-shift-estimate-small-shift} to $q^0$, and using
\eqref{eq:q-shift-comparison-small-shift} in $q(t)$ with $h=g$ and $x=y^0$, we get, for every $\eta>0$,
\begin{equation}\label{eq:shift-base-estimate-small-shift}
\int_0^T q^0(t)u(t)\dd t
\ge
-C_\eta
-\eta\int_0^T q(t)^2\dd t
-\eta\int_0^T u(t)^2\dd t,
\end{equation}
after decreasing $\eta$ if necessary. It remains to estimate $q-q^0$. By the decomposition of $\sigma$,
\[
q(t)-q^0(t)
=
m(t)\bigl(s(y(t))-s(y^0(t))\bigr)
+
r_\sigma(t,y(t))-r_\sigma(t,y^0(t)).
\]
Using \eqref{eq:s-shift-comparison-small-shift},
\eqref{eq:s-q-comparison-small-shift}, \eqref{eq:q-shift-comparison-small-shift}
and Young's inequality, we obtain
\begin{equation}\label{eq:principal-shift-error-small-shift}
\left|
\int_0^T
m(t)\bigl(s(y(t))-s(y^0(t))\bigr)u(t)\dd t
\right|
\le
\eta\int_0^T u(t)^2\dd t
+
C_\eta\alpha_\varepsilon^2\int_0^T q(t)^2\dd t
+
C_\eta.
\end{equation}
Similarly, with $\theta>0$ being the parameter in
\eqref{eq:r-negligible}, by Young's inequality, we get
\[
\left|
\int_0^T
\bigl(r_\sigma(t,y(t))-r_\sigma(t,y^0(t))\bigr)u(t)\dd t
\right|
\le
\eta\int_0^T u(t)^2\dd t
+
C_\eta\int_0^T
\bigl|r_\sigma(t,y(t))-r_\sigma(t,y^0(t))\bigr|^2\dd t .
\]
Using \eqref{eq:r-negligible}, we get
\[
\bigl|r_\sigma(t,y(t))-r_\sigma(t,y^0(t))\bigr|^2
\le
C\theta^2\bigl(|s(y(t))|^2+|s(y^0(t))|^2\bigr)
+
C_\theta .
\]
By \eqref{eq:s-q-comparison-small-shift} and
\eqref{eq:q-shift-comparison-small-shift}, since $\|g\|_\infty\le1$,
\[
|s(y(t))|^2+|s(y^0(t))|^2
\le
C\bigl(1+q(t)^2\bigr).
\]
Therefore,
\[
\left|
\int_0^T
\bigl(r_\sigma(t,y(t))-r_\sigma(t,y^0(t))\bigr)u(t)\dd t
\right|
\le
\eta\int_0^T u(t)^2\dd t
+
C_\eta\theta^2\int_0^T q(t)^2\dd t
+
C_{\eta,\theta}.
\]
Choosing $\theta>0$ small enough so that $C_\eta\theta^2\le\eta$, and
renaming the constant, we obtain
\begin{equation}\label{eq:remainder-shift-error-small-shift}
\left|
\int_0^T
\bigl(r_\sigma(t,y(t))-r_\sigma(t,y^0(t))\bigr)u(t)\dd t
\right|
\le
\eta\int_0^T q(t)^2\dd t
+
\eta\int_0^T u(t)^2\dd t
+
C_\eta.
\end{equation}
Combining \eqref{eq:shift-base-estimate-small-shift},
\eqref{eq:principal-shift-error-small-shift}, and
\eqref{eq:remainder-shift-error-small-shift}, we obtain
\[
\int_0^T q(t)u(t)\dd t
\ge
-C_\eta
-\bigl(C\eta+C_\eta\alpha_\varepsilon^2\bigr)\int_0^T q(t)^2\dd t
-C\eta\int_0^T u(t)^2\dd t.
\]
We first choose $\eta>0$ small enough and then choose
$\alpha_\varepsilon>0$ small enough so that
\[
C\eta+C_\eta\alpha_\varepsilon^2\le\varepsilon,
\quad
C\eta\le\varepsilon.
\]
The desired estimate follows after increasing the constant and denoting it by
$C_\varepsilon$.
\end{proof}

\section*{Acknowledgements}
The authors are grateful to Paul Gassiat for his presentation of his work on
the rough Bergomi model, which motivated this research. They also thank him for
his insightful remarks and suggestions, which significantly contributed to
improving this work. The authors also acknowledge the use of AI-assisted tools during the preparation
of this manuscript. 

\bibliographystyle{plainnat}
\bibliography{fusion_corrige}
\end{document}